\documentclass[11pt]{article}
\usepackage{amssymb}
\usepackage{amsmath}
\usepackage{amsthm}
\usepackage{etoolbox}
\usepackage{graphicx}
\usepackage[numbers]{natbib}
\usepackage[usenames,dvipsnames,svgnames,table]{xcolor}
\usepackage[pdfborder={0 0 0}]{hyperref}
\usepackage[textwidth=16cm, textheight=24.2cm]{geometry}

\usepackage[english]{babel}
\usepackage[utf8]{inputenc}
\usepackage[T1]{fontenc}
\usepackage{todonotes}
\usepackage{xcolor}

\newtheorem{thm}{Theorem}[section]
\newtheorem{cor}[thm]{Corollary}
\newtheorem{lem}[thm]{Lemma}

\newtheorem{prop}[thm]{Proposition}
\newtheorem{defn}[thm]{Definition}
\newtheorem{remark}[thm]{Remark}
\newtheorem{example}[thm]{Example}

\usetikzlibrary{matrix}
\usetikzlibrary{calc}
\tikzset{node style ge/.style={circle, font=\small}}

\def\rk{\operatorname {rank}}

\def\dim{\operatorname{dim}}

\def\ker{\operatorname{ker}}

\def\im{\operatorname{im}}

\def\id{\operatorname{id}}

\def\Spec{\operatorname{Spec}}

\def\Tor{\operatorname{Tor}}

\DeclareFontFamily{OMX}{MnSymbolE}{}
\DeclareSymbolFont{largesymbolsss}{OMX}{MnSymbolE}{m}{n}
\DeclareFontShape{OMX}{MnSymbolE}{m}{n}{
    <-6>  MnSymbolE5
   <6-7>  MnSymbolE6
   <7-8>  MnSymbolE7
   <8-9>  MnSymbolE8
   <9-10> MnSymbolE9
  <10-12> MnSymbolE10
  <12->   MnSymbolE12}{}
\DeclareFontShape{OMX}{MnSymbolE}{b}{n}{
    <-6>  MnSymbolE-Bold5
   <6-7>  MnSymbolE-Bold6
   <7-8>  MnSymbolE-Bold7
   <8-9>  MnSymbolE-Bold8
   <9-10> MnSymbolE-Bold9
  <10-12> MnSymbolE-Bold10
  <12->   MnSymbolE-Bold12}{}
\SetSymbolFont{largesymbolsss}{bold}{OMX}{MnSymbolE}{b}{n}
\DeclareMathSymbol{\lsem}{\mathopen}{largesymbolsss}{'102}
\DeclareMathSymbol{\rsem}{\mathclose}{largesymbolsss}{'107}
\def\formal#1{\lsem #1 \rsem}
\def\ideal{\lhd}

    \def\pp#1{\left( #1 \right)}
    \def\ann#1{\operatorname{ann}\pp{#1}}
    \def\annn#1#2{\operatorname{ann}_{#1}\pp{#2}}

    \def\tensor{\otimes}
    
    \def\into{\hookrightarrow}

    \def\Zzero{\mathbb{Z}_{\geq 0}}
    \def\gr{\operatorname{gr}}
    
    \def\Hilb{\operatorname{Hilb}}
    \def\Homop{\operatorname{Hom}}
    \def\Homt#1#2#3{\Homop_{#1}\left( #2, #3 \right)}
    \def\Hom#1#2{\Homop\left( #1, #2 \right)}
    
    \def\Homsht#1#2#3{\mc{H}om_{#1}\left( #2, #3 \right)}
    \def\QQ{\mathbb{Q}}
    \def\Gl#1{\Gamma\left( #1 \right)}
    
    \def\hook{\lrcorner}
    \def\rk{\operatorname{rk}}
    \def\dimk{\rk_k}
    
\begin{document}

    \def\DD#1#2{\global\csdef{D#1}{#2}#2}
    \def\mc#1{\mathcal{#1}}
    \def\OO#1{\mc{O}_{#1}}
    \def\TT#1{T_{#1}}

%    \vfill
%    \begin{center}
%    \Large\bfseries W~ramach inspiracji:
%    \includegraphics[width=16cm]{phd091411s}
%    \vfill
%    \includegraphics[width=12cm]{phd101212s}

%\end{center}

%    \vfill
%    \newpage

    \title{Deformations of zero-dimensional
    schemes and applications.\\ MSc thesis, adviser: Jaros\l{}aw Buczy\'nski}
    \author{Joachim Jelisiejew\thanks{Supported by the project
    ``Secant varieties, computational complexity, and toric degenerations''
    realised within the Homing Plus programme of Foundation
    for Polish Science,
    co-financed from European Union, Regional
    Development Fund.}}
    \maketitle

    \begin{abstract}
        In this thesis we consider the geometry
        of the Hilbert scheme of points in $\mathbb{P}^n$, concentrating on the
        locus of points corresponding to the Gorenstein subschemes of
        $\mathbb{P}^n$. New results
        are given, most importantly we provide tools for constructing flat
        families and analysis of finite Gorenstein algebras and expose their
        efficiency by proving smoothability of
        certain families of algebras. Much of the existing
        theory and folklore is reviewed, providing a micro-encyclopaedic reference.
    \end{abstract}

    {\footnotesize\noindent
    \textbf{keywords:} flatness, smoothability, finite Gorenstein
    schemes, Hilbert function decomposition, Hilbert scheme.}

    \section{Introduction}

    The Hilbert scheme of $r$-points of $\mathbb{P}^n$ parametrises closed
    zero-dimensional subschemes of $\mathbb{P}^n$ of degree $r$, or more
    precisely flat families of closed subschemes of degree $r$. It is one of the most
    important moduli spaces. Constructed by Grothendieck
    \cite{Gro} in 1961 it still draws much of attention because many
    natural questions about its structure are still open. It is known to be
    projective, by construction, and connected by a result of Hartshorne
    \cite{har66connectedness}. It is unknown exactly in which cases it is
    irreducible, though if $r$ and $n$ large, it is not, see \cite{CEVV}.
    An intriguing phenomena is that all proofs
    of reducibility are somehow indirect, resting on a dimension count on
    a tangent space or more sophisticated invariants. Much
    work is done to develop direct criteria, see~e.g.~\cite{ermanvelasco}.
    Similarly it
    is unknown whether the Hilbert scheme is reduced, though for large $r, n$ it is believed
    that it has arbitrarily bad singularities such as non-reduced irreducible
    components. 
    
    The problems with describing the exact structure of the Hilbert scheme of
    points had drawn attention to its important open subschemes such as the open
    \emph{Gorenstein locus}
    parametrising \emph{Gorenstein} subschemes of $\mathbb{P}^n$, see
    \cite{CN09}. Also this subset is known, for large $r, n$, to be
    reducible, see \cite{ia94}. In this paper we review much of the theory of
    zero-dimensional Gorenstein schemes and provide new tools for studying the
    Gorenstein locus in the aim to prove irreducibility of the locus for small
    $r$ and arbitrary $n$.

    This research is tightly connected with the study of secant and cactus
    varieties of a Veronese embedding $\nu:\mathbb{P}^n\to \mathbb{P}^N$: the
    zero-dimensional Gorenstein (resp.~smoothable Gorenstein) subschemes of
    $\mathbb{P}^n$
    of degree $r$ are used to parametrise $r$-th cactus (resp. $r$-th secant) variety of
    $\DD{verPn}{\nu\pp{\mathbb{P}^n}}$. See \cite{bubu2010}, especially Theorem 1.6 and Subsection 8.1, for details. 

    \subsection{Notation and text structure}

        All considered rings and algebras are commutative and with unity
        which is preserved by homomorphisms.
        Throughout the paper $k$ denotes a field and for a $k$-vector space
        $V$ the symbol $V^*$ denotes $\Homt{k}{V}{k}$; we will not use the star
        to refer to invertible elements of a ring. The expression $I\ideal A$
        means ``$I$ is an ideal of a ring $A$''. Where no confusion is likely
        to occur we use $(f_1,\dots,f_n)$ to denote the
        ideal generated by rings elements $f_1,\dots,f_n$. A \emph{finite} module
        (or vector space) is by definition a finitely generated module. We will use the
        word \emph{dimension} exclusively for Krull dimension. When speaking
        about linear dimension over $k$ we will use the words rank and corank instead
        of dimension and codimension and the symbol $\dimk$ instead of
        $\dim_k$. For example $\Spec k[\varepsilon]/\varepsilon^2$ has
        dimension zero and rank two over $k$.

        The text is divided into three expository sections and a
        section presenting original research results. Sections
        \ref{sec:flatness} and~\ref{sec:hilbscheme} recall the notions of
        flatness and Hilbert scheme. The emphasis is on
        flatness of projective and affine morphisms, anticipating our interest
        in finite morphisms. Section~\ref{sec:Gorenstein} is almost entirely algebraic or rather
        linear-algebraic, being concerned with zero-dimensional Gorenstein
        rings and their exploration through Macaulay's inverse systems. The paper
        culminates with the research Section \ref{ref:originalresearch:sec} analysing the geometry of the
        Gorenstein locus of Hilbert points in $\mathbb{P}^n$.

        Almost all results outside Section~\ref{ref:originalresearch:sec} are
        well-known or easy consequence of well-known results, however in many
        cases we did not find a suitable reference, thus being obliged to provide a proof.
        We have also provided proof of Theorem~\ref{ref:bjorkflatness:thm},
        whose original proof is complicated in the general setting of
        \cite{bjoerk}, and proofs of results from Subsection
        \ref{ref:localhilbertfunction:subsec} which more or less follow
        \cite{ia94}. Moreover all the results, with the notable exception of
        Theorem~\ref{ref:bjorkflatness:thm} and its supporting theory, were previously used in the
        setting and are considered standard.
        On the other hand all results contained in
        Section~\ref{ref:originalresearch:sec} are, as far as we know,
        original. They are put merely as an illustration of the introduced
        theory. Some of the techniques used can be found
        in \cite{JJ1551}, and a more thorough treatment will be
        given in a joint paper with
        Gf.~Casnati and R.~Notari.

    \section{Flatness}\label{sec:flatness}

    In this section we recall some standard material on flatness, mainly
    following \cite[Section III.9]{HarAG}.

    \begin{defn}
        Let $A$ be a ring and $M$ be an $A$-module. We say that $M$ is
        a \emph{flat} $A$-module if the functor $(-)\tensor_A M$ is exact.
    \end{defn}

    This definition comes with a geometric counterpart, see Lemma
    \ref{ref:basicflat:lem} below for the connection:
    \begin{defn}
        Let $f:X\to Y$ be a morphism of schemes and let $\DD{F}{\mc{F}}$ be a
        quasi-coherent sheaf on $X$. We say that $\DF$ is flat over $Y$ at a point $x\in X$
        if the stalk $\DF_x$ is a flat $\OO{Y,y}$-module, where $y = f(x)$ and
        we consider $\OO{X, x}$ as an $\OO{Y, y}$-module via the natural map. We say
        that \emph{$\DF$ is flat over $Y$} if it is flat at every point of
        $X$. We say that the morphism $X\to Y$ is flat if and only if $\OO{X}$ is flat
        over $Y$.
    \end{defn}

    Below we recall some basic properties of flatness, for simplicity the first
    two are stated only for
    flat morphisms:
    \begin{lem}\label{ref:basicflat:lem}
        \begin{enumerate}
            \item Flatness is stable under base change: for flat $X\to Y$
                and any $Y'\to Y$ the morphism $X\times_Y Y' \to Y'$ is flat;
            \item composition of flat morphisms is flat;
            \item let $A$ be a ring and $M$ be an $A$-module. The sheaf  
                $\DF = {\tilde{M}}$  on the affine scheme $Y = \Spec A$ is
                flat over $Y$ (with respect to $\id:Y\to Y$) if and
                only if $M$ is a flat $A$-module.
        \end{enumerate}
    \end{lem}
    \begin{proof}
        See \cite[Theorem III.9.2]{HarAG}.
    \end{proof}

    \subsection{Criteria for flatness}

    We will now investigate the flatness of modules.
    There are numerous criteria for flatness, we recall some of
    them below.

    \begin{lem}\label{ref:fingen:lem}
        Let $M$ be an $A$-module. If every finitely generated $A$-submodule
        of $M$ is flat over $A$ then $M$ is flat over $A$.
    \end{lem}

    \begin{proof}
        \def\colim{\varinjlim}
        The module $M$ is a direct limit of its finitely generated
        submodules $M_n$. Take any short exact sequence $\mc{S}$ of
        $A$-modules, then
        $M\tensor_A \mc{S}  \simeq  \pp{\colim M_n} \tensor_A \mc{S}  \simeq
        \colim \pp{M_n\tensor_A
        \mc{S}}$ and this is a short exact sequence of $A$-modules because it is a filtered direct limit of short exact
        sequences of $A$-modules.
    \end{proof}

%    \begin{proof}
%        Take a short exact sequence of $A$-modules $0\to N'\to N\to
%        N''\to 0$ and apply $(-)\tensor_A M$ obtaining
%        \[
%        0\to N'\tensor M \to N\tensor M\to N''\tensor M\to.
%        \]
%        This sequence is exact except, possibly, $0\to N'\tensor M\to N\tensor
%        M$. Suppose that there is a non-zero $m$ in the kernel of $N'\tensor M\to
%        N\tensor M$. From the construction of the tensor product it follows
%        that there if a finitely generated $A$-module $M'\subseteq M$ such that $m$ is the kernel of
%        $N\tensor M'\to N\tensor M$. This proves.
%        \textbf{Smuta!}
%    \end{proof}

    \begin{lem}
        A finitely generated module over a noetherian local ring is flat if
        and only if it is
        free. 
    \end{lem}

    \begin{proof}
        See \cite[Thm 3.2.7]{WeibHA}.
    \end{proof}

    \begin{lem}\label{ref:pidflatness:lem}
        Flat modules are torsion free.
        A module over a principal ideal domain (PID) is flat if and only if
        it is torsion free. In particular a submodule of a flat module over a PID is a
        flat module.
    \end{lem}
    \begin{proof}
        See \cite[Cor 6.3]{EisView} and
        Lemma~\ref{ref:fingen:lem}. Note that we do not assume
        that the module is finitely generated.
    \end{proof}

    We will be interested in connections of flatness and the behaviour of
    fibers of morphism. Here it is important to note, that we should make some
    additional assumptions on the morphism:
    \begin{example}\label{ref:hyperbola:exa}
        The hyperbola \[\Spec \frac{k[x, t]}{(tx - 1)} \to  \Spec k[t]\] is
        flat over $\Spec k[t]$, as $k[x, t]/(tx - 1)  \simeq  k[x,
        x^{-1}]$ is torsion free over $k[x^{-1}]$, but on the other
        hand the fiber over $(t-\alpha)$ for an invertible $\alpha\in k$ is
        isomorphic to a point $\Spec k$, and over $(t)$ the fiber is
        empty.

        On the other hand families with well-behaved fiber
        lengths are not necessary flat. For example we can
        add the origin to the aforementioned hyperbola, obtaining
        \[\Spec \frac{k[x, t]}{(tx - 1)\cap (t, x)} \to
        \DD{line}{\Spec k[t]},\]
        which has fibers over all $(t-\alpha)$ for $\alpha\in k$
        isomorphic to $\Spec k$. But this morphism is not flat!
        In fact $t\cdot (tx - 1) \in (tx - 1) \cap (t, x)$ and $tx-1\not\in
        (t, x)$, thus $t$ is a zero-divisor on $k[x, t]/\pp{(tx - 1)\cap
        (t, x)}$, contradicting Lemma~\ref{ref:pidflatness:lem}. Geometrically
        this corresponds to the fact that the origin and the hyperbola are
        different connected components.
    \end{example}

    \subsection{Flatness for coherent sheaves on projective space}

    Example~\ref{ref:hyperbola:exa} shows that, as usually when thinking about fibers, one should assume
    that the morphism is proper. Since ultimately we are interested in finite
    morphisms, which are projective, we will now consider only the stronger
    notion of projective morphism.
    
    \begin{defn}
        Let $T$ be a noetherian scheme and $\DD{F}{\mc{F}}$ be a coherent
        sheaf on $\DD{PTn}{\mathbb{P}_T^n}$.
        Define, for any point $t\in T$, the \emph{Hilbert polynomial} $\DD{Pt}{P_t}$ of the
        sheaf $\DF$ by the property
        \[
        \DPt(m) = \rk_{k(t)} H^0\big( \left(\mathbb{P}_{T}^n\right)_t,
        \DF_t(m) \big)\mbox{
        for }m\mbox{ large enough,}
        \]
        where $\DF_t$ is the scheme-theoretic fiber of $\DF$ over the point $t\in T$.
        If $\DF = \OO{Z}$ is the structure sheaf of a
        projective scheme $Z$, then this agrees with the usual definition of the
        Hilbert polynomial of $Z$.
    \end{defn}

    By vanishing of higher cohomologies the definition is correct, see
    \cite[III.5 Ex 5.2]{HarAG}.

    \begin{thm}\label{ref:Harflat:thm}
        Let $T$ be an integral noetherian scheme and $\DD{F}{\mc{F}}$ be a
        coherent sheaf on $\DD{PTn}{\mathbb{P}_T^n}$.
        Then the following are equivalent:
        \begin{enumerate}
            \item the sheaf $\DF$ is flat over $T$;
            \item the Hilbert polynomial $\DPt$ is independent of the choice of
                $t\in T$.
        \end{enumerate}
    \end{thm}
    \begin{proof}
        See \cite[Thm III.9.9]{HarAG}.
    \end{proof}

    We are mostly interested in the case when the fibers of $\DF$ are finite,
    so the Hilbert polynomial has degree zero.

    \begin{remark}\label{ref:degree:remark}
        If $\DF$ is flat and $\DPt$ is a polynomial of degree zero, then is it
        equal to $\rk_{\DD{kt}{k(t)}} \DD{GlobalF}{\Gl{\DF_t}}$. This
        number is called the \emph{rank of sheaf $\DF$ at $t\in T$}.
    \end{remark}

    The typical case is $\DF$ being the structure sheaf
    of a projective scheme $X \subseteq \DPTn$ such
    that $X\to T$ is flat and with finite fibers. If $T$ is
    of finite type over a field, then we have a useful
    corollary of Theorem~\ref{ref:Harflat:thm}:

    \begin{cor}\label{ref:lookatclosed:cor}
        Let $T$ be an integral scheme of finite type over a field and $\DF = \OO{Z}$ be the
        structure sheaf of a scheme $Z$ projective over $T$.
        Suppose that for every \emph{closed} $t\in T$ the rank of $\DF$ is finite and
        does not depend on the choice of $t$. Then $\DF$ is flat over $T$.
    \end{cor}

    \begin{proof}
        By Theorem~\ref{ref:Harflat:thm} and Remark~\ref{ref:degree:remark} it
        is enough to show that $\rk_{\Dkt} \DGlobalF$ is independent
        of the choice of point (not necessarily closed) $t\in T$. This is true
        for closed points, by assumption.

        Note that any open subset of $T$ contains a closed point.
        By semicontinuity of rank (see \cite[II.5 Ex 5.8a]{HarAG}), the set of $t\in T$ such that
        the rank is minimal is open, thus it
        contains a closed point. But now it follows that it contains all
        closed points, so in fact it is equal to $T$.
    \end{proof}

    \subsection{Flatness for filtered modules}

    Let $V$ be a closed subscheme of $\DD{Ank}{\mathbb{A}_k^n}$. In this
    section we will
    investigate flatness of $V$ under a projection $\DAnk \to
    \DD{Anone}{\mathbb{A}^1_k}$, making use of the gradation on $\DAnk$. Of
    course the gradation may not give a gradation on the algebra of global functions on
    $V$, but it always gives a filtration.

    \begin{defn}[Filtration]
        Let $\DD{A}{A}$ be a ring. A \emph{filtration} of $\DA$ is an infinite sequence
        of abelian subgroups of $A$:
        \[
        A_0 \subseteq A_1 \subseteq \dots \subseteq A_n \subseteq \dots
        \]
        such that
        \begin{enumerate}
            \item $1\in A_0$ and $A_n\cdot A_m \subseteq A_{n+m}$ for all $n,
                m\in \Zzero$;
            \item $\bigcup A_n = A$.
        \end{enumerate}
    \end{defn}

    \begin{defn}[Rees algebra of a filtration]
        If $A$ is a ring filtered by $A_n$, then the abelian subgroup
        \[
        \DD{AA}{\mc{A}} = \DAA(A_n) = \bigoplus_{n\in \Zzero} A_n t^n \subseteq A[t]
        \]
        of the polynomial ring $A[t]$ is actually a subring, called the
        \emph{Rees algebra of the filtration}, see \cite[Section
        6.5]{EisView}. We define the \emph{associated graded ring of $\DAA$} as
        \[
        \DD{gA}{\gr \DAA} = \DAA/(t).
        \]
        As an $A_0$-module $\DgA$ is isomorphic to $\bigoplus_{n\geq 0} A_n/A_{n-1}$, where
        $A_{-1} = 0$.
        Note that $A = \DAA/(t-1)$.
    \end{defn}

    \begin{example}\label{ref:gradedtofiltered:example}
        \begin{enumerate}
            \item Every ring $A$ can be filtered \emph{trivially} by letting $A_n
                = A$ for all $n\in \Zzero$, then the associated graded ring is
                just $A$.
            \item If $A = A^0\oplus A^1 \oplus \dots$ is a $\Zzero$-graded ring,
        then we obtain a filtration $A_n := \sum_{m\leq n} A^{m}$. In
        this situation we have a natural isomorphism of $A_0$-algebras $A
        \simeq \gr \mc{A}$ coming from $A_n/A_{n-1}  \simeq  (A^0 + \dots +
        A^n)\ /\ (A^0 + \dots +
        A^{n-1})  \simeq A^n$.
        \end{enumerate}
    \end{example}

    \begin{defn}[Filtration on a module]
        Let $\DA$ be a filtered ring and $\DD{M}{M}$ be an $\DA$-module.
        A~\emph{filtration} on $\DM$ is an infinite sequence
        of abelian subgroups of $\DM$:
        \[
        M_0 \subseteq M_1 \subseteq \dots \subseteq M_n \subseteq \dots
        \]
        such that $A_n\cdot M_m \subseteq M_{n+m}$ for all $n, m\in \Zzero$
        and $\bigcup_n M_n = M$.
        To a filtered module $M$ we can associate the Rees module
        \[
        \DD{MM}{\mc{M}} = \DMM(M_n) = \bigoplus_{n\in \Zzero} M_n t^n \subseteq M[t]
        \]
        which is naturally an $\DAA$-module, and the \emph{associated graded
        module} $\DD{gM}{\gr \DMM} = \DMM/(t)\cdot \DMM$ which is naturally a
        $\DgA$-module.
        Note that for any $N\in \Zzero$ the subgroup $\DMM_{\geq N} := \bigoplus_{n\geq N}
        M_n t^n \subseteq \DMM$ is an $\DAA$-submodule of $\DMM$.
    \end{defn}
        When speaking about the filtered rings and modules we will use the
        convention that $\mc{X}$ is the Rees algebra/module of an
        algebra/module $X$.
    One consequence of the above definitions is the following
    corollary.
    \begin{cor}\label{ref:gradedhassamelengthfinite:cor}
        Let $\DA$ be a filtered ring and $\DD{B}{B} \subseteq A_0$ be a subring.
        Suppose that $\DM$ is a filtered $\DA$-module. If
        $\DM$ is finitely generated over $\DA$ then $\DgM$ and $\DM$ have the
        same $\DB$-length (possibly infinite).
    \end{cor}

    \begin{proof}
        Recall that $B$-length is additive on short exact sequences of
        $B$-modules by the Jordan-H\"older theorem, see \cite[Thm 2.13]{EisView}.
        Fix a natural number $\DD{nzero}{n}$.
%        Since $\DM$  is finitely generated, there exists  such that
%        $\DM_{\Dnzero} = \DM$, then
%        \[\DgM = \bigoplus_{0\leq n\leq {\Dnzero}} \DM_n/\DM_{n-1}\ t^n\]
%        is a decomposition of $\DgM$ into $\DB$-modules.
        The decomposition of $\DM$ via the exact sequences
        \begin{align*}
            0\to M_{{\Dnzero}-1} \to& M
        \to M/M_{{\Dnzero}-1} \to 0,\\ 0\to M_{{\Dnzero}-2}\to& M_{{\Dnzero}-1}\to
        M_{{\Dnzero}-1}/M_{{\Dnzero}-2}\to 0,\\&\dots\\ 0\to M_0 \to& M_1 \to M_1/M_0\to 0
        \end{align*}
        proves that if for infinitely many natural $n$ the quotients
        $M_{n}/M_{n-1}$ are non-zero then the $B$-lengths of $\DM, \DgM$ are
        infinite. Consequently, we may assume that there exists $\Dnzero$ such that $M_{\Dnzero} =
        M_{\Dnzero + 1} = \dots = M$. Then the above decomposition
        proves that $\DM$ and $\DgM$ have equal $B$-lengths.
    \end{proof}

    An important thing about the associated graded modules is that an exact
    sequence of filtered modules gives rise to an \emph{exact} sequence of
    associated graded modules:
    \begin{prop}\label{ref:exactoffiltered:prop}
        Let $k$ be a field and $A$ be a filtered $k$-algebra.
        Suppose that filtered $A$-modules $M, N, P$ form an exact sequence
        \[
        0\to N\to M \to P\to 0,
        \]
        such that the filtration $N_n$ is the preimage of $M_n$ and $P_n$ is
        the image of $M_n$ for all $n\in \Zzero$.
        This sequence gives rise to exact sequences
        \[
        0\to \mc{N}\to \mc{M} \to \mc{P}\to 0\quad \mbox{  and  }\quad 0\to \gr \mc{N}\to \gr\mc{M}\to \gr\mc{P}\to 0.
        \]
    \end{prop}

    \begin{proof}
        The existence and exactness of the first sequence follows from
        assumptions. For the second one, note that $\mc{P} \subseteq P[t]$ is a torsion free
        $k[t]$-module, so it is flat. Now apply $(-)\tensor_{k[t]} k[t]/t$ obtaining an exact sequence
        \[
        0\to \mc{N}/t\mc{N}\to \mc{M}/t\mc{M}\to \mc{P}/t\mc{P}\to 0.\qedhere
        \]
    \end{proof}

    The following proposition computes the associated graded in the simplest
    case -- for a cyclic submodule of a graded module.
    \begin{prop}\label{ref:initialideals:prop}
        Let $A = \bigoplus A^n$ be a $\Zzero$-graded ring and $M = \bigoplus
        M^n$ be a $\Zzero$-graded $A$-module. Suppose that $(s) \ideal M$
        is generated by an element \[s = s_1 + \dots + s_n,\] where $s_i\in
        M^i$ and $s_n$ satisfies $\annn{A}{s_n} \subseteq \annn{A}{s}$.
        The ideal $(s)$ has a filtration coming from the filtration on $M$
        given as
        in the Example~\ref{ref:gradedtofiltered:example}, and, with respect to
        this filtration,\[\DD{grs}{\gr(s)}  \simeq (s_n).\]
    \end{prop}

    \begin{proof}
        The isomorphism $i:\gr \mc{M}  \simeq M$ from Example
        \ref{ref:gradedtofiltered:example} sends  ${\overline{m_0 + \dots +
        m_n}\in M_n/M_{n-1}}$, where $m_j\in M^j$, to $m_n$. Clearly
        $\DD{igrs}{i\big(\Dgrs\big)}$ is a submodule
        of $M$
        containing $(s_n)$. Take any $l\in \Zzero$ and a non-zero element
        $\overline{a\cdot s}\in
        M_l/M_{l-1}$. Let $a'$ be the homogeneous component of $a$ of maximal
        degree such that $a'\cdot s_n \neq 0$. From $\annn{A}{s_n} \subseteq
        \annn{A}{s}$ it follows that $a'\cdot s_n$ is
         the leading coefficient of $a\cdot s$. Under $i$ the element $a\cdot
         s$ is mapped to $a'\cdot s_n \in (s_n)$, which proves $\Digrs \subseteq
        (s_n)$.
    \end{proof}

    The presented methods, although very elementary, are quite handy when
    dealing with families of finite  algebras, as the following example shows:
    \begin{example}
        Let $k$ be an algebraically closed field, then $\Spec k[x, y, t]/(x^2 + t\cdot y\cdot x, y^2)\to \Spec k[t]$ is
        flat.
            
            \normalfont Indeed if we take the grading on $k[x, y, t]/y^2$ by powers of $x$,
            then the top degree form of $x^2 + t\cdot y\cdot x$ is equal to
            $x^2$ and thus independent of $t$.
            It follows that the fiber over each closed point $t\in T$ has
            $k$-rank equal to $\dimk k[x, y]/(x^2, y^2)$ and by Corollary~\ref{ref:lookatclosed:cor} we
            get that the morphism is flat.
    \end{example}

    A powerful generalization is given by the following theorem. Note that
    there are no assumptions neither on the ring nor on
    the module, in particular the module does not have to be finite.
    \begin{thm}\label{ref:bjorkflatness:thm}
        Let $A$ be a filtered ring and $M$ be a filtered $A$-module. If $\DgM$ is flat over $\DgA$, then
        $M$ is flat over $A$.
    \end{thm}

    \begin{proof}[Sketch of proof]
        This is a special case of \cite[Prop 3.12]{bjoerk}.

        We will prove the theorem with an additional
        assumption that $A$ is an algebra over a field $k \subseteq A_0$.
        Choose a free resolution $\DD{mcP}{\mc{P}_{\bullet}}$ of the Rees
        module $\DD{mcM}{\mc{M}}$ in the category of $\DAA$-modules. Since
        $\DmcM \subseteq M[t]$ is flat over
        $\DD{line}{k[t]}$ by Lemma~\ref{ref:pidflatness:lem} it
        follows that $\DmcP \tensor_{\Dline} \Dline/t$ is a free resolution of
        $\gr \DmcM$. Choose any $I\ideal A$ and equip it with the induced
        filtration, obtaining the Rees module $\DD{mcI}{\mc{I}}$.
        Note that $\DmcI$ is flat
        over $\Dline$, again by Lemma~\ref{ref:pidflatness:lem}.
        Since $\gr \DmcM$ is a flat $\gr \DD{mcA}{\DAA}$-module it follows that $\gr
        \DmcI\tensor_{\gr \DmcA} (-)$ is exact on $\DmcP
        \tensor_{\Dline} \Dline/t$, but as $\gr \DmcA = \DmcA/t$ this shows that
        \begin{equation}\label{eqn:maincomplex}
             \DmcI\tensor_{\DmcA} \DmcP \tensor_{\Dline} \Dline/t
             \simeq  \gr \DmcI\tensor_{\gr \DmcA} \pp{\DmcP \tensor_{\Dline}
             \Dline/t}
            \end{equation} is exact.

        The module $\DmcM$ is flat over $k[t]$, thus $\DmcM \tensor \DmcI
        \subseteq \DmcM$ is also flat. Similarly for any $i$ the module $\DmcI\tensor_{\mc{A}} \DD{mcPi}{\mc{P}_i}$
        is flat over $\Dline$, in particular
        multiplication by $t$ is injective on $\DmcI\tensor_{\mc{A}} \DmcPi$.
        From \eqref{eqn:maincomplex} it follows that the homology modules
        $\DD{mcHi}{\mc{H}_i}$ of
        $\DmcI \tensor_{\mc{A}} \DmcPi$ satisfy $t\cdot \DmcHi = \DmcHi$. This
        means that $\im d_{i+1} + t\cdot \ker d_i = \ker d_i$, where
        $d_{i}:P_{i}\to P_{i-1}$. Since $\im d_{i+1}, \ker d_i$ are
        filtered $A$-modules one checks directly that $\DmcHi =
        0$, thus $\DmcPi \tensor_{\mc{A}} \DmcI$ is exact. Now the complex $\DmcP
        \tensor_{\mc{A}} \DmcI\tensor_{\Dline} \Dline/(t-1)  \simeq
        \DmcP/(t-1)
       \tensor_{A} I$ is exact and $A$-modules $\DmcPi/(t-1)$ form a free
       resolution of $\DmcM/(t-1) = M$,
        hence we have $\Tor_1^{A}(I, M) = 0$ and this is a
        sufficient condition for $A$-flatness of $M$ by \cite[Prop 3.2.4]{WeibHA}.
    \end{proof}

    \begin{example}
        \begin{enumerate}
            \item The family
                \[\Spec k[x, y, t]/(x^2 + t\cdot y\cdot x)\to \Spec k[t]\mbox{
                is flat.}\]

                \normalfont Indeed, we choose a trivial grading on $k[t]$,
                grading by powers of $x$ on $k[x, y, t]$ and use Propositions
                \ref{ref:exactoffiltered:prop}, \ref{ref:initialideals:prop} to compute $\gr k[x, y, t]/(x^2
                + t\cdot y\cdot x)  \simeq k[x, y, t]/(x^2)$ which is clearly flat over
                $k[t]$. Note that the grading preserves $k[t]$ and so this is an
                isomorphism of $k[t]$-modules.
            \item Suppose we would like to compute \[\gr \DD{II}{(x^2 - 1,
                xy -1)},\] where $\DII\ideal k[x, y]$, with respect
                to the grading by total degree. Then \[x - y = y\cdot (x^2 -1) -
                x\cdot (xy-1)\notin (x^2, xy)\] is an
                element of $\gr \DII$, so the leading forms of $x^2 - 1, xy-1$
                do not generate $\gr \DII$.
                This suggests that $x^2 - 1, xy -1$ is not a ``good'' generating
                set; a ``better one'' is $x-y, x^2 - 1$ because $\gr
                \DII = (x - y, x^2)$.
                One can see a connection with Gr\"obner bases.
        \end{enumerate}
        
    \end{example}

    \section{The Hilbert scheme}\label{sec:hilbscheme}

    In this section we follow a very accessible introduction by Str\"omme \cite{Stromme}.

    One advantage of flat projective families is that there exists a good
    scheme parametrising them, known as the Hilbert scheme. We need some
    definitions.
    \begin{defn}[Hilbert functor]
        Let $X$ be projective an $S$-scheme. Define a functor
        \[\DD{HilbX}{\underline{\Hilb}_{X/S}}:Sch_S^{op} \to Set\] by putting, for any $S$-scheme $T$,
        \[
        \DD{HilbXT}{\underline{\Hilb}_{X/S}(T)} := \left\{ \mbox{ closed subschemes }Z
        \subseteq X \times_S T \mbox{ such that the projection }Z \to T\mbox{ is flat }\right\}
        \]
        and for any morphism $T'\to T$ the map
        $\DHilbXT\to \DD{HilbXTT}{\DHilbX(T')}$ sending $Z \subseteq X\times_S
        T$ to $Z' = Z \times_T T'$.
        In a similar way for any polynomial $\DD{P}{P}\in \QQ[x]$ we can
        define a (sub)functor
        $\DD{HilbXP}{\underline{\Hilb}_{X/S}^P}:Sch_S^{op} \to Set$ by letting
        \begin{align*}
            {\DHilbXP(T)} := \big\{ &\mbox{ closed subschemes }Z
            \subseteq X \times_S T \mbox{ such that the projection }Z \to
            T\mbox{ is flat }\\&\mbox{ and the Hilbert polynomial of
            }\left(\OO{Z}\right)_t \mbox{ is
            equal to }P\mbox{ for any }t\in T\big\},
        \end{align*}
        see Theorem~\ref{ref:Harflat:thm}.
    \end{defn}

    \def\DHSchX{\Hilb_{X/S}}
    \begin{thm}\label{ref:HilbertExistence:thm}
        Suppose that $X$ is projective over a noetherian scheme $S$.
        For any polynomial $P\in \QQ[x]$ the functor $\DHilbXP$ is
        represented by a \emph{projective} scheme $\DD{HSchXP}{\DHSchX^{P}}$ and the functor
        $\DHilbX$ is represented by a scheme $\DHSchX$ equal to
        the countable disjoint union of all
        schemes $\DHSchXP$ where $P\in \QQ[x]$ (some of which are empty). We
        call this scheme the Hilbert scheme of $X/S$.
    \end{thm}
    \begin{proof}
        See \cite[Thm 1.1]{HarDeform}.
    \end{proof}

    One important consequence of representability is the fact that $S$-points
    of $\DHSchX$ correspond to closed subschemes of $X$. For a closed
    subscheme $Z \subseteq X$ by $[Z]\in \DHSchX$ we denote the corresponding $S$-point.

    Not much is known about the local geometry of the Hilbert scheme; for
    example in many cases we do not know if this scheme is reduced. However,
    we have the following result by Grothendieck, computing the tangent space
    at a point:
    \begin{thm}
        Let $k$ be a field, $X$ be projective over $\Spec k$ and $Z
        \subseteq X$ be a closed subscheme given by the ideal sheaf
        $\DD{Ish}{\mc{I}}$.
        The tangent space to $\DD{HSchXk}{\Hilb_{X/k}}$ at the point $[Z]$ is isomorphic to
        $H^{0}\left( Z, \DD{Nsh}{\mc{N}_{Z/X}} \right)$, where $\DNsh$ is the
        normal sheaf of $Z$ in $X$.
    \end{thm}
    \begin{proof}
        By classifying flat families over $k[\varepsilon]/\varepsilon^2$; see
        \cite[Thm I.2.4]{HarDeform}.
    \end{proof}

    Being interested in finite schemes, we will use only a local
    consequence of this result:
    \begin{cor}\label{ref:tangentspaceaffinely:cor}
        Suppose that $k$ is a field, $X$ is projective over $\Spec k$ and let $Z\subseteq X$
        be a closed scheme with support contained in an open affine subscheme
        $U = \Spec A$ of $X$, where the subscheme $Z$ is cut out by an ideal
        $I\ideal A$. Then the tangent space to $\DHSchXk$ at $[Z]$ is
        isomorphic to $\Homt{A/I}{I/I^2}{A/I}$.
    \end{cor}

    \begin{proof}
        Let $\DIsh$ be the ideal sheaf of $Z\subseteq X$.
        Recall that $\DNsh = \Homsht{\OO{X}}{\DIsh/\DIsh^2}{\OO{X}/\DIsh}$ may be
        regarded as a coherent sheaf on $X$, by \cite[Remark II.8.9.1]{HarAG}. Clearly, its
        global sections are equal to its sections over $U$:
\[H^{0}\left( X, \DD{Nsh}{\mc{N}_{Z/X}} \right) = H^0\left( U,
\Homsht{\OO{X}}{\DIsh/\DIsh^2}{\OO{X}/\DIsh} \right)  \simeq
\Homt{A}{I/I^2}{A/I}\simeq  \Homt{A/I}{I/I^2}{A/I},\]
because on the affine scheme $U$ we have the equivalence of categories of
quasi-coherent sheaves and modules.
    \end{proof}

    \subsection{The smoothable component}

    Fix a natural number $r$ and a projective variety $X$ over an algebraically
    closed field $k$.
    In this subsection we investigate the scheme $\DD{HSchXr}{\DHSchXk^r},$ where we
    think of $r\in \mathbb{Z}$ as a constant polynomial. The scheme
    $\DHSchXr$ is commonly referred as to ``the Hilbert scheme of $r$ points of
    $X$''. As the name suggests, the simplest example of a zero-dimensional
    subscheme of degree $r$ is a disjoint union of $r$ reduced points of
    $X$. The closure of points of $\DHSchXr$ corresponding to such subschemes will be called
    the smoothable component, see Definition~\ref{ref:defsmoothable:def}. The
    main goal of this subsection is to show that it is indeed an irreducible
    component of $\DHSchXr$.

    \begin{lem}\label{ref:tangentspaceatsmooth:lem}
        Let $\{p_1,\dots,p_r\} = Z \subseteq X$ be a disjoint union of $r$ closed points; then the tangent
        space to $\DHSchXr$ at the point $[Z]$ has $k$-rank equal to the sum
        of $k$-ranks of tangent spaces to $X$ at $p_1, p_2,\dots,p_r$.
    \end{lem}

    \begin{proof}
        \def\pm{\mathfrak{m}}
        Since $X$ is projective, we can find an affine open subset $\Spec A$
        of $X$
        containing $p_1, \dots, p_r$. By Corollary
        \ref{ref:tangentspaceaffinely:cor} it is enough to compute
        $\dimk \Homt{A/I}{I/I^2}{A/I}$
        where $I\ideal A$ is the (finitely generated) ideal of $Z$;
        in fact $I
        = \pm_1\cap \dots\cap \pm_r$ where $\pm_1,\dots,\pm_r$ are ideals
        corresponding to
        $p_1,\dots,p_r$. The module $I/I^2$ is supported at $p_1,\dots,p_r$, so
        \begin{align*}
        \dimk\Homt{A/I}{I/I^2}{A/I} &= \sum_{1\leq i\leq r}
        \dimk\Homt{(A/I)_{\pm_i}}{(I/I^2)_{\pm_i}}{(A/I)_{\pm_i}}&\\ &=\sum_{1\leq
        i\leq r} \dimk\Homt{k}{\pm_i/\pm_i^2}{k},
    \end{align*}
    since $\pp{A/I}_{\pm_i}  \simeq  A/\pm_i \simeq k$ for every $i$.\qedhere
    \end{proof}

    Now we will describe an important open subset of $\DHSchXr$.
    Let $X_i = X$ for $i=1,2,\dots,r$ and set
    \[\DD{Xtimesr}{X^r} := X_1 \times X_2 \times \dots \times
    X_r\]
    with projections $\pi_i : X^r \to X_i$ for $i=1,\dots,r$.
    Denote by $\Delta \subseteq X^r$ the union of all pullbacks of diagonals $\Delta_{ij}
    \subseteq X_i \times X_j$, where $1\leq i<j\leq r$.
    Let $X_{smooth}$ be the set of smooth points of $X$ and
    set $U = X_{smooth}^r\setminus \Delta$; this is an open set in $X^r$.
    Let $Z_i \subseteq X^{r} \times X$ be the pullback of the
    diagonal $\Delta_{i} \subseteq X_i \times X$ and take
    \[
    Z := Z_1 \cup \dots \cup Z_r,\quad Z_U = Z \times_{X^r} U.
    \]

    \begin{prop}\label{ref:flatmorphism:prop}
        The family $Z_U \to U$ is flat over $U$ and thus gives a morphism
        \[
        \varphi: U\to \DHSchXr.
        \]
    \end{prop}

    \begin{proof}
        Since $Z\to X$ is projective and $Z_U\to U$ is its pullback, so it
        is a projective morphism. The
        fiber over a closed point $(x_1, x_2, \dots, x_r)\in U$ is just
        $\OO{\{x_1\}\cup\dots\cup\{x_r\}}$, so the morphism is quasi-finite
        with fibers of rank $r$,
        thus finite by \cite[Cor 12.89]{WedhornAG} 
        and $\OO{{Z_U}}$ is a coherent sheaf over $U$.
        Since $U$ is open in the variety $\DXtimesr$ it is integral, so from
        Corollary~\ref{ref:lookatclosed:cor} it follows that $Z_U \to U$ is
        flat.
    \end{proof}

    \begin{prop}
        The morphism $\varphi: U\to \DHSchXr$ defined in Proposition
        \ref{ref:flatmorphism:prop} is flat, in
        particular open. The
        image of $\varphi$ has dimension $r\cdot \dim X$.
    \end{prop}

    \begin{proof}
        Denote by $\DD{HSchXrzero}{\overline{\im \varphi}}$ the
        scheme-theoretic image of $\varphi$.
        The morphism $\varphi$ has finite fibers over closed points, thus 
        the dimension of $\overline{\im \varphi}$ is not less than $\dim U$,
        by \cite[Thm 10.10]{EisView}.
        Since $\DHSchXrzero$ is projective and
        irreducible, its
        dimension is equal to the
        dimension of the local ring at any point of $\DHSchXrzero$.
        Take $u\in U$. By Lemma~\ref{ref:tangentspaceatsmooth:lem} we see that the tangent
        spaces at $u$ and $\varphi(u)$ have the same $k$-rank. In
        particular
        \begin{multline*}
            \dim U \leq \dim \DHSchXrzero = \dim \OO{\DHSchXrzero,
            \varphi(u)} \leq \dim\DD{OOHu}{\OO{\DHSchXr, \varphi(u)}} \leq \dimk \TT{\DHSchXr, \varphi(u)} =
            \dimk \TT{U, u} = \dim U,
        \end{multline*}
        as $U$ is smooth. This shows that $\DHSchXr$ is smooth at
        $\varphi(u)$. Now the result follows as a special case of \cite[Thm
        18.16]{EisView}.
    \end{proof}

    Now we define, for a projective variety $X$, the smoothable
    component.
    \begin{defn}\label{ref:defsmoothable:def}
        Let $\DD{HSchXrzero}{{\DHSchXr}^\circ}$ be the irreducible component of the Hilbert scheme
        containing the image
        of the morphism $\varphi: U \to \DHSchXr$ defined in
        Proposition~\ref{ref:flatmorphism:prop}. We call $\DHSchXrzero$ the
        \emph{smoothable component} of the Hilbert scheme $\DHSchXr$. It has dimension $r\cdot \dim X$.
    \end{defn}

    \subsection{Smoothability}

    Let $X$ be as in the previous subsection.
    One of the central questions of the theory of Hilbert schemes of points is
    ``for which $r$ it is true that $\DHSchXr = \DHSchXrzero$?'' and its
    refinements. To give some down-to-earth conditions equivalent to
    this equality, we introduce the notion of smoothing.
    \begin{defn}
        A (flat) \emph{deformation} is a flat family $Y\to B$, where
        $B$ is irreducible; we call $B$ the base of deformation.

        Let $R$ be a closed subscheme of $X$. A
        \emph{smoothing of $R$ in $X$} is a closed subscheme $Y \subseteq B
        \times X$ such that the morphism $Y\to B$ is proper and a flat
        deformation with fiber $Y_b = \{b\}\times R$ at some closed point $b\in B$ and
        $Y_{\eta}$ smooth at the generic point $\eta\in B$.
        We say that \emph{$R$ is smoothable in $X$} if there exists a smoothing
        of $R$ in $X$.
    \end{defn}

    Intuitively the above definition says: to deduce that $[R]\in \DHSchXrzero$
    we find a morphism $B\to \DHSchXr$ such the generic point of $B$ maps into
    the interior of $\DHSchXrzero$ and a special point maps to $[R]$. The
    condition of being proper is technical, connected with the fact that
    the Hilbert scheme is defined only for projective $X$, but important --- otherwise it
    would be easy to~e.g.~deform two points to only one point or one point to an empty scheme, see Example~\ref{ref:hyperbola:exa}.

    \begin{remark}\label{ref:punctualimportant:remark}
        The existence of
        smoothings $Y_1\to B_1, \dots, Y_n\to B_n$ of connected
        components of $R$ implies the existence of a smoothing of $R$ over
        $B_1\times \dots \times B_n$. This leads
        one to the conclusion that when
        considering smoothings of finite
        $R$ it is enough to consider those $R$ which are supported at a single point.
    \end{remark}

    The following proposition together with
    Proposition~\ref{ref:smoothabilityeverywhere:prop} shows a connection
    between smoothability and the smoothable component:
    \begin{prop}\label{ref:smoothabilityvssmoothings:prop}
        Let $X = \DD{Pkn}{\mathbb{P}_k^n}$ and $R\subseteq X$ be
        a zero-dimensional closed subscheme of degree $r$. The following
        conditions are equivalent
        \begin{enumerate}
            \item $[R]\in \DHSchXrzero$;
            \item there exists a smoothing of $R$ in $X$;
            \item there exists a smoothing of $R$ in $X$ over a base $D =
                \DD{Disk}{\Spec k\formal{x}}$.
        \end{enumerate}
    \end{prop}

    \begin{proof}
        This is \cite[Lem 4.1]{CEVV}.
    \end{proof}

    It is natural to ask what is the class of schemes $Z$ such that $R$ is smoothable
    in $Z$. The following lemma shows that this class is in some sense
    directed; this
    answers \cite[Question 2.12]{bgl2010}.
    \begin{lem}
        Let $\varphi:Y_1\to Y_2$ be an affine morphism of schemes and suppose that
        $R\subseteq Y_1$ is a zero-dimensional subscheme of $Y_1$ such that $\im \varphi(R)  \simeq R$. If $R$ is
        smoothable in $Y_1$, then it is smoothable in $Y_2$.
    \end{lem}

    \begin{proof}
        Let $D = \DDisk$ and suppose that $Z \subseteq Y_1 \times D$ is a
        smoothing of $R$ in $Y_1$. We claim that $Z' := \im (\varphi \times
        \id)(Z) \subseteq Y_2 \times D$ is a smoothing of $\im \varphi(R)$ in
        $Y_2$. It would suffice to show that $\varphi' = \varphi \times \id_D$ gives an isomorphism
        of $Z$ and $Z'$ as schemes over $D$. Let $\DD{Dclosed}{t}\in D$
        denote the closed point of $D$.
        Since $R = Z_{\DDclosed}, \im \varphi(R) = Z'_{\DDclosed}$ are finite and $Z\to D$ is proper, it follows
        that $Z$ is finite over $D$, then $Z'$ is also finite over $D$. In
        particular $Z, Z'$ are affine. Now $Z_{\DDclosed}  \simeq
        Z'_{\DDclosed}$ and
        $Z$ is flat over $D$, so the result follows from Nakayama's
        lemma (see \cite[Cor 4.8]{EisView}).
    \end{proof}

    At the end, we note the following important result from \cite{bubu2010},
    which uses and generalizes \cite[Lem 2.2]{CN09}:
    \begin{prop}\label{ref:smoothabilityeverywhere:prop}
        Suppose that $X$ and $Y$ are two projective varieties over an
        algebraically closed field of characteristic zero. Let $R$ be a
        zero-dimensional scheme. Fix closed embeddings $R\subseteq X$ and
        $R\subseteq Y$. If $R$ is smoothable in $Y$ and $R \subseteq X$ is
        supported in the smooth locus of $X$, then $R$ is smoothable in $X$.
    \end{prop}
    \begin{proof}
        See \cite[Prop 2.1]{bubu2010}.
    \end{proof}

    It is important to know that there are successful attempts of proving
    smoothability of zero-dimensional schemes by parametrising the Hilbert
    scheme and using computer algebra systems, see \cite{bertone2012division}.

    \section{Gorenstein schemes and algebras}\label{sec:Gorenstein}

    This large section is devoted to proving various properties of Gorenstein
    schemes, needed for the later considerations on the points of the Hilbert
    scheme.

    In this section we will consider only zero-dimensional
    noetherian schemes. They are affine, so we will talk of
    rings $A$ rather than schemes $\Spec A$.
%    In the introduction we mainly follow
%    \cite[Chapter 21]{EisView}.
    Let us first state the preliminaries needed
    for the definition of a zero-dimensional Gorenstein module.
    We mainly follow \cite[Chapter 21]{EisView}.

    If $A$ is a ring, $M$ is an $A$-module and $I\ideal A$ then we denote
    \[\ann{M} := \left\{ a\in A\ |\ aM = 0 \right\},\qquad \annn{M}{I} := \left( m\in M\
    |\ Im = 0 \right).\]

    \begin{remark}
        Finitely generated modules over a zero-dimensional noetherian ring are
        artinian.
    \end{remark}

    \begin{defn}\label{ref:socledef:def}
        Let $(A,\DD{mm}{\mathfrak{m}}, k)$ be a zero-dimensional local ring and $M$ be a
        finitely generated $A$-module. Define \emph{the socle} of $M$ to be
        the annihilator of $\Dmm$. We denote this submodule by
        $\DD{socM}{\operatorname{soc}(M)} = \annn{M}{\Dmm}$.
    \end{defn}

     Recall that
    a~submodule $N \subseteq M$ is an \emph{essential submodule of $M$} if
    $N\cap M'\neq 0$ for any submodule $0\neq M'\subseteq M$.
    The following lemma justifies the use of the word ``socle''; it is a
    standard exercise in the theory of modules, see~e.g.~\cite[2.4 Ex.
    18]{LamFC}.
    \begin{lem}\label{ref:socleissocle:lem}
        Under the assumptions of Definition~\ref{ref:socledef:def}, the socle
        of $M$ is the smallest, with respect to inclusion, essential submodule
        of $M$.
    \end{lem}

    \begin{proof}
        First, let us show that $\DsocM$ is essential. Take any nonzero $M'
        \subseteq M$. The sequence $M' \supseteq \Dmm M'\supseteq \Dmm^2 M'
        \supseteq \dots$ is eventually constant: there exists $n$ such that
        $\Dmm^n M' = \Dmm^{n+1} M'$ and by Nakayama's lemma (see
        \cite[Cor 4.8]{EisView}) we get $\Dmm^n M' = 0$.
        Let $l\geq 0$ be the maximal natural number such that $\Dmm^l M' \neq
        0$, then $\Dmm^l M' \subseteq M'\cap \DsocM$.

        Now, if $m\in \DsocM$, then $Am$ is a simple module, so any essential
        submodule of $M$ contains it.
    \end{proof}

        The proof of Lemma~\ref{ref:socleissocle:lem} motivates the following
        definition:
    \begin{defn}\label{ref:Gorensteinsocledegree:def}
        Let $(A,\Dmm, k)$ be a zero-dimensional local ring and $M$ be a
        finitely generated $A$-module. We say that $M$ is \emph{Gorenstein}
        if and only if $\dimk\DsocM = 1$.
        If $M$ is Gorenstein then the \emph{socle degree of $M$} is the
        maximal $l$ such that $\Dmm^{l} M \neq 0$.
    \end{defn}

    \begin{remark}
        Since $\DsocM  \simeq \Homt{A}{k}{M}$ and any zero-dimensional
        local ring is (trivially) Cohen-Macaulay this definition agrees with the
        usual definition of Gorenstein rings.
        There is an important connection between Gorenstein rings and duality
        theory, see \cite[Section 21.3]{EisView}.
    \end{remark}

    \begin{prop}\label{ref:injectivityofGorenstein:prop}
        Let $(A,\Dmm, k)$ be a zero-dimensional local ring. Then the following
        are equivalent
        \begin{enumerate}
            \item $A$ is Gorenstein,
            \item $A$ is injective as an $A$-module.
            \item the canonical module $\omega_{A}$ is isomorphic to
                $A$ as an $A$-module.
        \end{enumerate}
    \end{prop}

    \begin{proof}
        The canonical module is defined as a consequence of \cite[Prop 21.1]{EisView}
        and the above statement is a part of \cite[Thm 21.5]{EisView}.
    \end{proof}

    For arbitrary rings being Gorenstein is defined as a stalk-local property:
    \begin{defn}
        Let $A$ be a zero-dimensional ring and $M$ be a finitely generated
        $A$-module. Then $M$ is a \emph{Gorenstein $A$-module} if and only if $M_{\mathfrak{m}}$ is a
        Gorenstein $A_{\mathfrak{m}}$-module for all maximal ideals
        $\mathfrak{m}$ of $A$.
    \end{defn}

    \subsection{Gorenstein points of the Hilbert scheme}

    Let $X$ be a projective scheme over a field $k$.

    \begin{lem}\label{ref:Gorensteinisopen:lem}
        Let $X\to Y$ be a flat, quasi-finite morphism of projective schemes.
        Then the set of points $y\in Y$ such that the fiber $X_y$ is
        Gorenstein is open in $Y$.
    \end{lem}

    \definecolor{light-gray}{gray}{0.85}
    \newcommand{\mybox}[1]{%
    \noindent\colorbox{light-gray}
    {\parbox{\textwidth}{#1}}}

    \begin{proof}
        \def\DX{X}
        \def\DY{Y}
        Unfortunately, a detailed exposition would take us to far from our main
        subject, we only sketch the proof.
        The morphism $\DX\to \DY$ is projective and quasi-finite, thus
        finite by \cite[Cor 12.89]{WedhornAG}.
        The relative dualizing sheaf $\DD{omegacos}{\omega_{X/Y}}$ of a morphism $X\to Y$ is introduced in
        \cite[Def 6]{Klei} and it exists for $\DX\to \DY$ by \cite[Cor
        18]{Klei}. Furthermore $\Domegacos$ is a quasi-coherent sheaf over $\DX$ by definition and
        coherent by finiteness of the morphism and the properties from
        \cite[Def 1, Prop 9]{Klei}. Now the locus where $\Domegacos$ is
        not invertible is closed in $\DX$, thus the locus of $y\in \DY$ such
        that $\pp{\Domegacos}_{y}$ is invertible is open in $\DY$. The sheaf
        $\Domegacos$ is stable under base change by \cite[Prop 9]{Klei}, thus
        $\pp{\Domegacos}_{y}  \simeq \omega_{X_y/y}$ for every $y\in Y$.

        Let $k$ be a field and $A$ be a finite $k$-algebra. Note that we have
        a canonical isomorphism $\Homt{A}{M}{\Homt{k}{A}{W}}  \simeq
        \Homt{k}{M}{W}$ for every $A$-module $M$ and $k$-module $W$.
         Let $X' = \Spec
        A$ and $Y'=\Spec k$. The relative dualizing sheaf $\DD{tmpom}{\omega_{X'/Y'}}$ exists and it is
        isomorphic as a sheaf over $X'$ to the sheafification of
        $\Homt{k}{A}{k}$ by \cite[Def 1,
        Prop 9]{Klei} and
        the mentioned canonical isomorphism.
        The $A$-module $\Homt{k}{A}{k}$ is isomorphic to the canonical
        module $\omega_A$,
        as defined in \cite[Prop 21.1]{EisView}, by the
        discussion after \cite[Cor 21.3]{EisView}. Now the claim follows from
        Proposition \ref{ref:injectivityofGorenstein:prop}.
    \end{proof}

    As an immediate consequence of Lemma~\ref{ref:Gorensteinisopen:lem} we see that
    there is a good locus of the Hilbert scheme parametrising Gorenstein
    subschemes of $X$.
    \begin{prop}\label{ref:GorensteinlocusinHilbr:prop}
        Let $\pi:\DD{tmpU}{\mathcal{U}}\to \DD{tmpH}{\DHSchXr}$ be the projection from the universal
        family. The set of points $h\in \DtmpH$ such that the
        fiber of $\pi$ over $h$ is Gorenstein is open in $\DtmpH$.
%        and let $U \subseteq \DtmpU$ be the open set of points where
%        $\DD{omegacos}{\omega_{\DtmpU/\DtmpH}}$, the relative dualizing sheaf,
%        is invertible. We call the projection of $U$ to $\DtmpH$
%        the \emph{Gorenstein locus} of $\DHSchXr$ and denote it by
%        $\DD{HSchXrG}{\operatorname{HilbGor}_{X/k}^{r}}$. The $k$-points of
%        $U$ correspond to the fibers which are Gorenstein, i.e. to the Gorenstein schemes
%        of length $r$ in $X$.
    \end{prop}

    The tangent space to the Hilbert scheme at a $k$-point corresponding to a
    Gorenstein subscheme is also easier to calculate, thanks to the
    injectivity:
    \begin{prop}\label{ref:Gorensteintangentspace:prop}
        Let $Z\subseteq_{cl} X$ be a Gorenstein zero-dimensional closed scheme
        with support contained in an open affine subscheme
        $U = \Spec A$ of $X$, where the subscheme $Z$ is cut out by an ideal
        $I\ideal A$. Then the tangent space to $\DHSchXk$ at $[Z]$ has
        rank $\dimk I/I^2 = \dimk A/I^2 - \dimk A/I$.
    \end{prop}

    \begin{proof}
        We can analyse the connected components of the support of $Z$
        separately, thus we assume that $A/I$ is local.
        By Proposition~\ref{ref:injectivityofGorenstein:prop} we see that the
        functor $\Homt{A/I}{-}{A/I}$ is exact. Now by induction on the
        length we prove that $\dimk \Homt{A/I}{M}{A/I} = \dimk M$ for
        any finitely generated $A/I$-module $M$.
    \end{proof}

    \subsection{Macaulay's inverse systems}

    From now on, we will assume that $(A, \Dmm, k)$ is a finite local
    $k$-algebra such that $k\to A/\Dmm$ is an isomorphism. We will say simply
    ``let $(A, \Dmm, k)$ be a finite local $k$-algebra'', implicitly assuming
    that $k\to A/\Dmm$ is an isomorphism.

    So far we did not give any examples of Gorenstein rings. Macaulay's
    inverse systems give a rich source of such examples, to some extent
    classifying all local Gorenstein zero-dimensional $k$-algebras.
    First, we explain how to ``intrinsically'' put an algebra structure on the space of
    $k$-functionals on a linear space $V$, when $V$ has an additional structure
    of a coalgebra. Since we will be mainly interested in the example
    $V = k[x_1,\dots,x_n]$ and the result will be, at least in characteristic
    zero, $V^*  \simeq k\formal{\frac{\partial}{\partial
    x_1},\dots,\frac{\partial}{\partial x_n}}$, the
    less curious reader should skip directly to Theorem~\ref{ref:Macaulaycorrcor:thm}.

    We recall the definition of a coalgebra, for simplicity we deal
    only with cocommutative coalgebras:
    \begin{defn}\label{ref:coalgebra:def}
        A $k$-module $C$ with $k$-linear
        maps $\DD{diag}{\Delta}: C\to C\tensor C$ \emph{(comultiplication)} and $\varepsilon:
        C\to k$ \emph{(counity)} is a \emph{(cocommutative)
        $k$-coalgebra} if the following equations (dual to the
        commutative monoid equations) are satisfied:
        \begin{enumerate}
            \item $(\varepsilon \tensor \id_C)\circ \Ddiag = (\id_C \tensor \varepsilon)
                \circ \Ddiag = \id_C\mbox{  (counitality)}$,
            \item $(\id_C \tensor
                \Ddiag)\circ \Ddiag = (\Ddiag \tensor \id_C) \circ\Ddiag\mbox{
                (coassociativity)}$,
            \item $\DD{tmpswap}{\operatorname{swap}} \circ \Delta = \Delta
                \mbox{  (cocommutativity)}$, where
                $\Dtmpswap: C \tensor C \ni \sum c_1\tensor c_2 \mapsto \sum c_2\tensor
                c_1\in C\tensor C$.
        \end{enumerate}
    \end{defn}

    \begin{example}
        In this paper we only encounter coalgebras of the form $\Gamma(S,
        \OO{S})$ for an affine commutative monoid $S$, with $\Ddiag,
        \varepsilon$ induced by multiplication and identity of $S$. 
    \end{example}

    \begin{cor}\label{ref:algebrafromcoalg:cor}
        If $C$ is cocommutative coalgebra then the
        $k$-linear space
        $C^*:=\Hom{C}{k}$ with the multiplication defined by
        \[\pp{\varphi_1\cdot \varphi_2}(c) = \pp{\varphi_1\tensor
        \varphi_2}(\Ddiag c),\]
        is a commutative $k$-algebra and $C$ with multiplication
        \[
        \mu:C^* \tensor C \to C,\mbox{ defined by }\mu(\varphi\tensor c) =
        \pp{\varphi\tensor\id_C}(\Ddiag(c)),
        \]
        is a $C^*$-module. We will write $\varphi\cdot c$ for
        $\mu(\varphi\tensor c)$.
    \end{cor}

    \begin{proof}
        The $k$-algebra structure follows directly from Definition
        \ref{ref:coalgebra:def}. To check that $C$ is a $C^*$-module, it is,
        by linearity, sufficient to prove $\varphi_1\cdot (\varphi_2\cdot c) =
        (\varphi_1\cdot \varphi_2)\cdot c$ and this follows from definition of
        multiplication on $C^*$ and coassociativity of $C$.
    \end{proof}

    An enlightening exercise is to analyze dependencies between the
    subcoalgebras of $C$ and $C^*$-submodules of $C$.

    \begin{remark}
        We are primarily interested in the example $C =
        \DD{GlAn}{\Gamma\pp{\mathbb{G}_a(k)^n}} = k[x_1,\dots,x_n]$, where
        $\mathbb{G}_a(k)$ is the affine line over $k$ with addition.
        In this case the ring obtained in Corollary
        \ref{ref:algebrafromcoalg:cor} is called the ring of divided powers, but
        we will not go into details here.
        Below we give an explicit description of $C^*$ and its action on $C$
        when $k$ is a field of characteristic zero.
    \end{remark}

    \begin{prop}\label{ref:computedualofAn:prop}
        Let $k$ be a field of characteristic zero.
        Fix the isomorphism \[\DGlAn  \simeq
        k[x_1,\dots,x_n] := C,\]
        then $C^*$ is isomorphic to $k\formal{y_1, \dots,
        y_n}$,
        the ring of formal power series, by identifying a monomial $y_1^{a_1}\cdot
        \dots y_n^{a_n}$ with the dual to
        the monomial $\pp{a_1!\cdot a_2!\cdot\ldots\cdot a_n!}^{-1}\cdot x_1^{a_1}\dots x_n^{a_n}$ in the monomial basis.
        The element $y_i$ acts on $k[x_1,\dots,x_n]$ by
        $f\mapsto y_i\hook f := \frac{\partial f}{\partial x_i}$ and this extends to the action
        of $C^*$ by multiplicativity and countable additivity. 
    \end{prop}

    \begin{proof}
        \def\ii#1{\mathbf{#1}}
        We need the multi-index notation: let  $\ii{x} := (x_1, \dots, x_n)$,
        then~by $\ii{x}^{\ii{a}}$ we mean the monomial $\prod
        x_i^{a_i}$, by $\ii{a}!$ the product of $a_i!$ and by
        $\binom{\ii{a}}{\ii{b}}$ the expression
        $\ii{a}!/(\ii{b}!\pp{\ii{a}-\ii{b}}!)$, which is defined when $a_i\geq
        b_i$ for all $i$.

        Let us fix the isomorphism of $k$-vector spaces $C^* \simeq k\formal{z_1, \dots, z_n}$
        by identifying $\ii{z}^{\ii{a}}$ with the dual to $\ii{x}^{\ii{a}}$ in the
        monomial basis.
        We observe that $\Ddiag: C\to C\tensor C$ is defined by setting
        $\Ddiag(x_i) := 1\tensor x_i + x_i \tensor 1$ and extending to a
        $k$-algebra homomorphism, which gives
        \def\ind#1#2{\ii{#1}^{\,\ii{#2}}}%
        \def\indbinom#1#2{\binom{\ii{#1}}{\ii{#2}}}%
        \[
        \Ddiag\pp{\ind{x}{a}} = \sum_{\ii{b}, \ii{c}:\ \ii{b} + \ii{c} = \ii{a}}
        \indbinom{a}{b}\ind{x}{b} \tensor \ind{x}{c},\mbox{ so
        that }
        \]
        \[
        \ind{z}{b}\cdot \ind{z}{c} = \binom{\ii{b}+\ii{c}}{\ii{b}}\ii{z}^{\ii{b}+\ii{c}}.
        \]
        Now define a $k$-linear map $f:k\formal{y_1,\dots,y_n}\to
        C^*$ by $f\pp{\ind{y}{a}} = \ii{a}!\cdot \ind{z}{a}$.
        Since $k$ is of characteristic $0$ it is an isomorphism.
        From the above multiplication law it follows that $f$ is a $k$-algebra
        homomorphism.
        To finish the proof it is sufficient to note that
        \[
        y_i \hook \ind{x}{a} = \binom{\ii{a}}{1_i}\ii{x}^{\ii{a}-1_i} =
        a_i\cdot \ii{x}^{\ii{a}-1_i} = \frac{\partial \ind{x}{a}}{\partial x_i},
        \]
        where $1_i = (0, \dots, 0, 1, 0, \dots, 0)$ with $1$ on $i$-th
        position.
    \end{proof}

    Now we will prove the main theorem of Macaulay's inverse systems.
    The duality coming from this theorem is, as far as we know,
    used only in the case of $C = \DGlAn$, but we will
    prove it in the language of coalgebras. The reasons are twofold. Firstly, the proof is essentially
    the same and clearer to prove in the abstract setting. Secondly, even in
    the case $C = \DGlAn$, where $k$ has positive characteristic, the
    algebra $C^*$ is quite complicated to deal by hand.

    Recall that for a ring $A$ and $A$-modules $N \subseteq M$ we say that $N$ is a \emph{small}
    $A$-submodule of $M$ if for every $A$-submodule $M' \subseteq M$ the
    equality $M' + N = M$ implies $M' = M$. As we will see, this is
    ``dual'' to being essential.

    For a coalgebra $C$ and $D \subseteq C$ by $D^{\perp}$ we denote $\{\varphi\in
        C^{*}\ |\ \varphi(d) = 0\mbox{ for all }d\in D\}$, similarly for $I \subseteq C^*$ by
        $I^{\perp}$ we denote $\left\{ c\in C\ |\ i(c) = 0\mbox{ for all }i\in
        I\right\}$, where $\varphi(d)$ is the value of a functional
        $\varphi\in C^*$ on an element $d\in C$.

    \begin{thm}[Macaulay's inverse systems]\label{ref:Macaulaycorr:thm}
        Let $C$ be a cocommutative $k$-coalgebra.
        Suppose that
        \begin{enumerate}
            \item  there is
                a rank one subcoalgebra $C_0 \subseteq C$, contained in each
                non-zero $C^*$-submodule of $C$;
            \item\label{item:coalgebrasdense} for every ideal $I$ of $C^*$ such
                that $C^*/I$ is finite (over $k$) there is a finite  (over $k$) $C^*$-submodule $D \subseteq C$ such
                that $D^{\perp} \subseteq I$.
        \end{enumerate}
        Then
        \begin{enumerate}
            \item there is a bijection
                \begin{align*}
                \left\{ I\ideal C^*\ |\ \dimk C^*/I< \infty \right\}
                &\longleftrightarrow \left\{C^*\mbox{-submodules of }C\mbox{ finite over }k\right\}\\
                    f_1:I\ideal C^*\ \  &\longmapsto\ \  I^{\perp}\\
                    f_2:D^{\perp}\ideal C^*\ \ &\mbox{\reflectbox{$\longmapsto$}}\ \  D \subseteq C;
                \end{align*}
            \item the above bijection preserves $k$-rank:
                \[\dimk C^*/I = \dimk I^{\perp}\mbox{ and }\dimk D = \dimk
                C^*/D^{\perp};\]
            \item the algebra $C^*$ is local with residue field $k$;
            \item\label{item:essential} If $D, D' \subseteq C$ are finitely generated $C^*$-modules,
                then $D' \subseteq D$ is small if and only if the image of $D'^{\perp}$
                in $C^*/D^{\perp}$ is essential;
            \item the above bijection restricts to a bijection between
                Gorenstein quotients of $C^*$ and cyclic $C^*$-submodules of
                $C$.
        \end{enumerate}
    \end{thm}

    \begin{proof}
        We view $C$ as contained in $C^{**}$ via the canonical map.
%        Since $C^*/I$ is finite over $k$, the vector space $I^{\perp}
%        \cap C$ is finite over $k$ and it is a $C^*$-module, because
%        $I$ is an ideal. On the other hand, if $D\subseteq C$ is a subcoalgebra,
%        $D^{\perp}$ is an ideal of $C$. If $D \subseteq C$ is finite, then
%        $D^*$ is finite. This proves that $f_1, f_2$ are
%        well defined.
        \begin{enumerate}
            \item
                Clearly $D \subseteq f_1\circ f_2(D) = \left(D^{\perp\perp} \cap
                C\right)$ and if $c\in C \setminus D$ then we can find
                $\varphi\in D^{\perp}$ such that $\varphi(c)\neq 0$, this
                proves $f_1\circ f_2 = \id$.
                Similarly, $I \subseteq f_2\circ f_1(I) = \pp{I^{\perp} \cap
                C}^{\perp}$. Take $D \subseteq C$ from Condition~\ref{item:coalgebrasdense}, then we have a perfect pairing
                between finite vector spaces $C^*/D^{\perp}$ and
                $D$, which restricts to a perfect pairing of $C^*/I$ with
                $f_1(I)$; this proves $f_2\circ f_1(I) = I$ by counting
                $k$-ranks.
            \item It is sufficient to check this for $f_2$ and then it is
                trivial as $D$ is finite over $k$ and $C^*/D^{\perp} \simeq
                D^*$.
            \item As $C_0 \subseteq D \subseteq C$ for any $D$, we see that
                $C_0^{\perp}$ is the largest ideal of $C^*$ such that the
                quotient is finite over $k$, thus the largest
                ideal of $C^*$.
            \item Let $I := D^{\perp}$. Suppose that $\DD{fJ}{J/I} \subseteq
                \DD{fI}{C^*/I}$ is essential and set $\DD{dJ}{E} := f_1(J)
                \subseteq D$. Take any submodule $\DD{dJprim}{E'}
                \subsetneq
                D$ with corresponding ideal $\DD{fJprim}{J'/I} \subseteq C^*/I$.
                Since $\DfJprim\neq 0$ we have $\DfJprim \cap \DfJ \neq 0$ and
                so $\left(J' \cap J\right)^{\perp} = \DdJprim + \DdJ$ is not
                the whole of $D$. This argument can be reversed.
            \item Gorenstein quotients are precisely those containing a
                rank one essential submodule, so by Point~\ref{item:essential} it is enough to prove that cyclic modules
                are precisely those having a corank one small
                submodules. Indeed by Nakayama's lemma (see
                \cite[Cor 4.8]{EisView}) a cyclic module $C^*c$
                has corank one
                small submodule $\Dmm_{C^*}c$. Suppose now that $D\subseteq C$
                contains a corank one small submodule $D' \subseteq D$
                and choose $c\in D\setminus D'$; then $C^*c + D' = D$, so
                $C^*c = D$.\qedhere
        \end{enumerate}
    \end{proof}

    The corollary of Theorem~\ref{ref:Macaulaycorr:thm} are the classical
    Macaulay's inverse systems:
    \begin{thm}[Macaulay's inverse systems for
        $\mathbb{G}_a^n$]\label{ref:Macaulaycorrcor:thm}
        Let $k$ be a field of characteristic zero and $V$ be a $k$-vector
        space with basis $x_1,\dots,x_n$.
        Let $\DD{SS}{S} = \operatorname{Sym}(V) = \DD{Sring}{k[x_1,\cdots,x_n]}$ be a polynomial $k$-algebra and
        $\DD{TT}{S^*} := \DD{Tring}{k\formal{y_1,\cdots,y_n}}$ be a ring of power series. We view $\DTT$
        as acting on $\DSS$ by identifying $y_i$ with $\frac{\partial}{\partial
        x_i}$; denote this action by $\hook: \DTT\tensor \DSS\to \DSS$. There is a bijection
                \begin{align*}
                    \left\{I\ideal \DTT\ |\ \dimk \DTT/I < \infty \right\}
                    &\longleftrightarrow \left\{\mbox{ finitely generated }\DTT\mbox{-submodules of }\DSS\right\}\\
                    f_1:I\ideal\DTT\ \  &\longmapsto\ \  \annn{\DSS}{I} =: I^{\perp}\\
                    f_2:M^{\perp} := \annn{\DTT}{M}\ \ &\mbox{\reflectbox{$\longmapsto$}}\ \
                    M\subseteq \DSS;
                \end{align*}
                preserving $k$-rank and restricting to a
                bijection between Gorenstein quotients of $\DTT$ and cyclic $\DTT$-submodules of
                $\DSS$. Moreover $I$ is homogeneous if and only if $\annn{\DSS}{I}$ is
                generated by homogeneous polynomials.
            \end{thm}
    \begin{proof}
        We will use Theorem~\ref{ref:Macaulaycorr:thm} together with
        Proposition~\ref{ref:computedualofAn:prop}. First let us check the
        assumptions of Theorem~\ref{ref:Macaulaycorr:thm}.
        First, $k \subseteq \DSS$ is contained in any non-zero $\DTT$-submodule
        of $\DSS$. Next, take $I\ideal \DTT$ such that $\DTT/I$ is finite
        over $k$. The algebra $\DTT/I$ is local and
        artinian, thus $\Dmm_{\DTT}^{m+1} \subseteq I$ for some $m$. But
        $\Dmm_{\DTT}^{m+1}$ is
        the ideal orthogonal to the $\DTT$-submodule of $\DSS$ consisting of
        polynomials of degree at most $m$.
        
        Note that each cyclic $\DTT$-submodule of $\DSS$ is finite
        over $k$, so that a $\DTT$-submodule $M \subseteq \DSS$ is
        finite over $k$ if and only if it is a finitely generated $\DTT$-module.
        Now the main claims follow from Theorem~\ref{ref:Macaulaycorr:thm}
        together with Proposition~\ref{ref:computedualofAn:prop}, we left the homogeneity fact for the
        reader.
    \end{proof}

    \begin{defn}
        For an ideal $I\ideal \DTT$, any $f\in S$ such that $f^{\perp} = I$ is called
        the \emph{dual socle generator} of $\DTT/I$. Conversely for $f\in
        \DSS$ the quotient $\DTT/f^{\perp}$ is called the \emph{apolar
        algebra of $f\in \DSS$}.
    \end{defn}

    \begin{remark}
        Every local finite Gorenstein algebra $(A,\Dmm, k)$ may be viewed as a
        quotient of $\DTT$, and the most important part of
        Theorem~\ref{ref:Macaulaycorrcor:thm} is that
        any such algebra is isomorphic to an algebra of the form
        $\DTT/f^{\perp}$ for some polynomial $f\in \DSS$. We will explore this
        ideas in Subsection~\ref{ref:dualgen:section}.
    \end{remark}

    \begin{example}
        \def\Dpartial#1{\frac{\partial}{\partial x_{#1}}}
        If $f = x_1^2 + x_1\cdot x_2$, then $I = \left(
        \pp{\Dpartial{2}}^2, \pp{\Dpartial{1}}^2 - 2\cdot
        \Dpartial{1}\Dpartial{2} \right)$. The socle of
        the dual algebra is generated by the image of
        e.g.~$\Dpartial{1}\Dpartial{2}$.
    \end{example}

    \subsection{Perfect pairings}

    The following proposition states, in the linear-algebra language, that a finite local
    Gorenstein algebra $A$ is isomorphic to its canonical module, see
    Proposition~\ref{ref:injectivityofGorenstein:prop}). It will be used in
    the following subsection.

    \begin{prop}\label{ref:ppairingonGorenstein:prop}
        Let $(A, \Dmm, k)$ be a finite local Gorenstein $k$-algebra. Choose any splitting into
        $k$-modules $A = V \oplus \DD{socA}{\operatorname{soc}(A)}$ and a projection $\pi: A\to
        \DsocA  \simeq k$.
        Then the pairing
        \[
        m: A \times A \ni (a, b) \longmapsto \pi(ab)\in k
        \]
        is perfect.
    \end{prop}

    \begin{proof}
        It is enough to prove that for any $a\in A$ we have $\DD{tmpim}{m(\{ka\} \times
        A)}\neq 0$. But $\Dtmpim = \pi(Aa) \neq 0$ as the ideal $Aa$ contains
        an element from the socle by Lemma~\ref{ref:socleissocle:lem}.
    \end{proof}

    \begin{remark}\label{ref:orthogonaltoideal:sremark}
        Even though the choice of splitting and $\pi$ clearly affects the
        pairing, we
        see that for any ideal $I$ the orthogonal
        complement of $I$ is the annihilator $\ann{I}$, in particular it does not depend on these choices.
    \end{remark}

    \subsection{Local Hilbert function of a Gorenstein
    algebra}\label{ref:localhilbertfunction:subsec}

    In this and next subsections we follow a foundational paper of the theory
    of zero-dimensional Gorenstein $k$-algebras by Iarrobino \cite{ia94}. The
    idea is to exploit the perfect pairing from Proposition
    \ref{ref:ppairingonGorenstein:prop} to gain information about the powers
    of the maximal ideal and thus about the Hilbert function. The most
    important result is the decomposition of the Hilbert function, Theorem
    \ref{ref:iarrobinoHfdecomposition:thm}.

    \begin{defn}
        Let $(A, \Dmm, k)$ be a local ring. The \emph{local Hilbert
        function} of $A$ is defined by
        \[
        h_A: \Zzero \ni n \longmapsto \dimk \Dmm^n/\Dmm^{n+1}\in \Zzero.
        \]
        If $A$ is a finite $k$-algebra, then $h(n) = 0$ for $n\gg 0$ and we
        will identify the function with the vector of its nonzero values.
    \end{defn}

    Note that the local ring $(A, \Dmm, k)$ has a natural filtration by powers of
    maximal ideal $\Dmm$ and that the local Hilbert function is computed
    from $\gr \DAA$ with respect to this filtration.
    \def\Dmmperp#1{\left(\Dmm^{#1}\right)^{\perp}}%

    Let us denote $\Dmmperp{m} := \ann{\Dmm^{m}}$ for any $m\geq 0$ and adopt the
     convention that non-positive powers of ideals are equal to the whole ring.
    A direct consequence of Proposition~\ref{ref:ppairingonGorenstein:prop}
    with Remark~\ref{ref:orthogonaltoideal:sremark} is:
    \begin{cor}\label{ref:loewyHilbertfunc:cor}
        Let $(A, \Dmm, k)$ be a finite local Gorenstein
        $k$-algebra. Then for every $m\geq 0$ we have
        \[
        h_A(m) = \dimk
        \frac{\Dmmperp{m+1}}{\Dmmperp{m}}.
        \]
    \end{cor}

    \begin{proof}
        Indeed, fixing any pairing as in Proposition
        \ref{ref:ppairingonGorenstein:prop} we see that
        \[
        \pp{\frac{\Dmm^m}{\Dmm^{m+1}}}^*  \simeq
        \frac{\Dmmperp{m+1}}{\Dmmperp{m}}.\qedhere
        \]
    \end{proof}

    \def\Dhd#1#2{\Delta_{#1}\pp{#2}}
    \def\Dhdvect#1{\Delta_{#1}}
    \def\iaq#1{Q_{#1}}
    \def\iac#1{A_{#1}}
    Iarrobino noticed that a Gorenstein  algebra satisfies strong symmetry conditions
    on the local Hilbert function:
     \begin{thm}[Hilbert function decomposition]\label{ref:iarrobinoHfdecomposition:thm}
        Suppose that $(A, \Dmm, k)$ is a finite local Gorenstein
        $k$-algebra with socle degree (see~\ref{ref:Gorensteinsocledegree:def}) equal to $j$.
        Define
        \def\iainter#1#2{\Dmm^{#1}\cap \Dmmperp{#2}}%
        \def\iasum#1#2{\Dmm^{#1} + \Dmmperp{#2}}%
        \def\iacomp#1#2{\frac{\iainter{#1}{#2}}{\left( \iainter{#1}{#2}
        \right) \cap \left(\iasum{#1+1}{#2-1}\right)}}%
        \def\iacompbis#1#2{\frac{\iainter{#1}{#2}}{\iainter{#1}{#2-1} +
        \iainter{#1+1}{#2}}}%
        \[
        \iac{m, n} := \iainter{m}{n}\mbox{ and } \iaq{m, n} :=
        \frac{\iac{m, n}}{\iac{m+1, n} + \iac{m, n-1}}\mbox{ for all } 0\leq
        m, n\leq j+1.
        \]
        Denote
        \[\Dhd{A, s}{t} = \Dhd{s}{t} := 
        \begin{cases}\dimk \iaq{t, j+1 -
        (s+t)}\quad &\mbox{for}\quad
        0\leq s\leq j,\ 0\leq t\leq j-s,\\ 0 &\mbox{ otherwise}\end{cases}\]
        and call this the \emph{Hilbert function decomposition} with rows $\Dhdvect{s}$. It is
        convenient to note
        that $\dimk \iaq{m, n} = \Dhd{j+1 - (m+n)}{m}$.
        \begin{enumerate}
            \item $\iaq{m, 0} = \iaq{j+1, m} = 0$ for all $m$.
                If $m < j+1$, then $\iaq{0, m} = 0$.
            \item If $m + n > j + 1$, then $\iac{m, n} = \iac{m, n-1}$, so
                $\iaq{m, n} = 0$.
            \item\label{item:ppairing} Fix any pairing $A \times A \to k$ defined as in Proposition
                \ref{ref:ppairingonGorenstein:prop}.
                It induces a perfect pairing
                \[
                \iaq{m, n} \times \iaq{n-1, m+1} \to k,
                \]
                in particular $\Dhd{s}{t} = \Dhd{s}{(j-s)-t)}$, we say that $\Dhdvect{s}$ is symmetric with respect to $\frac{j-s}{2}$.
            \item For $j\geq a\geq 0$ denote $h_a(t) := \sum_{i=0}^{a}
                \Dhd{i}{t}$.
                The quotient
                \[\frac{\gr \DAA}{\DD{Ca}{C_a} = \bigoplus_{0\leq m\leq j}
                \DCa^{m}}\] has Hilbert
                function $h_a$, where $\DCa^m \subseteq \Dmm^m/\Dmm^{m+1}$ is the
                image of $\iac{m, j - (m+a)} = \iainter{m}{j - (m+a)}$.
            \item The local Hilbert function $h_A$ satisfies $h_A(t) = 0$
                for $t
                > j$ and
                \[
                h_A(t) = \sum_{i = 0}^{j-t} \Dhd{i}{t}\quad
                \mbox{for  } 0\leq t\leq j.
                \]
        \end{enumerate}
    \end{thm}
    \begin{remark}
        Some of the claims of the theorem may be summarized by saying that the
        following tables are respectively ``symmetric up to taking duals''
        (see Point \ref{item:ppairing} of the theorem) and symmetric with respect to their anti-diagonals:
        \begin{center}%
%            \vspace*{-2em}%
                \begin{tikzpicture}%
                    \matrix [matrix of math nodes, nodes = {node style ge},,column
                    sep=0 mm, row sep=0mm, ampersand replacement=\&]%
                    {%
                    \&\iaq{0, j+1} \& 0 \& \dots \& 0 \& \node(qrighttop){0};\&\phantom{mmm}\&\&\Dhd{0}{0} \& 0 \& \dots \& 0 \& \node (righttop){0};\\
                    \& \iaq{0, j} \& \iaq{1, j} \& 0  \& \dots \& 0\&\&\&0 \& \Dhd{0}{1}  \& 0  \& \dots \& 0\\
                    \&\dots\&\dots\&\dots\&\dots\&\dots\&\&\&\dots\&\dots\&\dots\&\dots\&\dots\\
                    \&\iaq{0, 2} \& \iaq{1, 2} \& \dots  \& \iaq{j-1, 2} \& 0\&\&\&0  \& \Dhd{j-2}{1} \& \dots  \& \Dhd{0}{j-1}  \& 0\\
                    \&\node (qleftbottom){\iaq{0, 1}}; \& \iaq{1, 1} \& \dots\&  \iaq{j-1, 1} \& \iaq{j, 1}\&\&
                    + \&\node (leftbottom){0}; \& 0 \& \dots \&  0 \& \node(rightbottom){\Dhd{0}{j}};\\
                    \&A/\Dmm\&\Dmm/\Dmm^2\&\dots\&\Dmm^{j-1}/\Dmm^j\&\Dmm^j\&\&
                    \& \node(lbb){h_A(0)}; \& h_A(1) \& \dots \& h_A(j-1) \& \node(rbb){h_A(j)};\\
                    };
                    \draw [dashed, ultra thick, orange] (qleftbottom) -- (qrighttop);
                    \draw [dashed, ultra thick, orange] (leftbottom) -- (righttop);
                    \draw ($(leftbottom)!0.5!(lbb) - (1.2, 0) $ ) -- ($
                    (rightbottom)!0.5!(rbb) + (0.8, 0) $);
                \end{tikzpicture}
%                \begin{tikzpicture}%
%                    \matrix [matrix of math nodes, nodes = {node style ge},,column
%                    sep=0 mm, row sep=0mm, ampersand replacement=\&]%
%                    {%
%                    \&\Dhd{0}{0} \& 0 \& \dots \& 0 \& \node (righttop){0};\\
%                    \&0 \& \Dhd{0}{1}  \& 0  \& \dots \& 0\\
%                    \&\dots\&\dots\&\dots\&\dots\&\dots\\
%                    \&0  \& \Dhd{j-2}{1} \& \dots  \& \Dhd{0}{j-1}  \& 0\\
%                    + \&\node (leftbottom){0}; \& 0 \& \dots \&  0 \& \node(rightbottom){\Dhd{0}{j}};\\
%                    \& \node(lbb){h_A(0)}; \& h_A(1) \& \dots \& h_A(j-1) \&
%                    \node(rbb){h_A(j)};\\
%                    };
%                    \draw [dashed, ultra thick, orange] (leftbottom) -- (righttop);
%                    \draw ($(leftbottom)!0.5!(lbb) - (1.2, 0) $ ) -- ($
%                    (rightbottom)!0.5!(rbb) + (0.8, 0) $);
%                \end{tikzpicture}
            \end{center}
        \end{remark}

        \begin{proof}[Proof of Theorem~\ref{ref:iarrobinoHfdecomposition:thm}]
        \def\iainter#1#2{\Dmm^{#1}\cap \Dmmperp{#2}}%
        \def\iasum#1#2{\Dmm^{#1} + \Dmmperp{#2}}%
        \begin{enumerate}
            \item Obviously $\iac{m, 0} = \iac{m, -1} = \Dmm^m$ and
                $\Dmm^{j+1} = 0$, this proves the first equality. For the
                second equality, note that if $m < j+1$, then $\Dmm^{m}
                \neq 0$  and so $\Dmmperp{m} \subseteq \Dmm$.
            \item By assumption $\Dmm^{m+n-1} = 0$, thus  $\iainter{m}{n} =
                \iainter{m}{n-1} = \Dmm^m$.
            \item Recall that $k$-modules are \emph{modular} i.e. if $M, N, P$
                are $k$-modules and $M \subseteq P$ then $(M + N)\cap P = M +
                N\cap P$. In particular, as for any $n, m$ we have $\Dmm^{n} \subseteq
                \Dmm^{n-1}$ and $\Dmmperp{m} \subseteq \Dmmperp{m+1}$ it
                follows that
                \begin{equation}\label{eqn:modularity}
                    \begin{aligned}
                        \pp{\iasum{n+1}{m-1}} \cap \Dmm^{n} = \Dmm^{n+1} +
                        \iainter{n}{m-1},\\\left(\Dmm^{n+1} +
                        \iainter{n}{m-1}\right) \cap \Dmmperp{m} = \iainter{n+1}{m} +
                        \iainter{n}{m-1}.
                    \end{aligned}
                \end{equation}

                Below, for an $A$-module $M$ by $M^*$ we denote
                $\Homt{k}{M}{k}$ and for $N \subseteq M$ by $N^{\perp}
                \subseteq M^*$ we denote functionals which are zero on $N$.
                The fixed perfect pairing $A\times A\to k$ gives an isomorphism of
                $k$-modules $A
                \simeq A^*$. From Remark~\ref{ref:orthogonaltoideal:sremark}
                it follows that
                \[\pp{\iainter{n-1}{m+1}}^{\perp}  \simeq
                \iasum{m+1}{n-1},\quad\mbox{and}\quad
                \pp{\iasum{n}{m}}^{\perp}  \simeq \iainter{m}{n},\mbox{ thus }\]
                \begin{align*}
                    \pp{\frac{\iainter{n-1}{m+1}}{\pp{\iainter{n-1}{m+1}}
                    \cap \pp{\iasum{n}{m}}}}^*  \simeq
                    \frac{\pp{\iainter{n-1}{m+1}}^{\perp}
                    +
                    \pp{\iasum{n}{m}}^{\perp}}{\pp{\iainter{n-1}{m+1}}^{\perp}}
                     \simeq \\
                      \simeq \frac{\pp{\iasum{m+1}{n-1}} +
                     \pp{\iainter{m}{n}}}{\iasum{m+1}{n-1}}  \simeq
                     \frac{\iainter{m}{n}}{\iainter{m}{n}\cap
                     \pp{\iasum{m+1}{n-1}}},
                \end{align*}
                which is, by Equation \eqref{eqn:modularity}, equivalent to
                $\iaq{n-1, m+1}  \simeq \iaq{m, n}^*$. Now $\Dhd{s}{t} = \dimk \iaq{t, j+1-(s+t)} =
                \iaq{j-(s+t), t+1} = \Dhd{s}{j-(s+t)}$.

            \item\label{item:subquotients} For any $m\in \Zzero$ we have $\Dmm\cdot \iac{m, j - (m+a)} \subseteq \iac{m+1, j
                - (m+1+a)}$ which proves that $\DCa$ is an ideal in
                $\gr \DAA$.
                Let us fix $e := j - (m+a)$ and look at the filtration
                \[
                \iac{m, e} \subseteq \iac{m, e+1} \subseteq \iac{m, e+2} \subseteq
                \dots \iac{m, j} \subseteq \iac{m, j+1} = \Dmm^m.
                \]
                Since $\Dmm^{m+1} \cap \iac{m,n} = \iac{m+1,n}$, this
                gives a filtration
                \[
                0 = \frac{\iac{m, e}}{\iac{m+1, e} + \iac{m, e}} \subseteq
                \frac{\iac{m, e+1}}{\iac{m+1, e+1} + \iac{m, e}} \subseteq \dots \subseteq
                \frac{\iac{m, j+1}}{\iac{m+1, j+1} + \iac{m, e}} =
                \frac{\Dmm^m}{\Dmm^{m+1} + \iainter{m}{e}},
                \]
                with associated graded $\iaq{m, e+1} \oplus \iaq{m, e+2} \oplus
                \dots \iaq{m, j+1}$. Since the modules are finitely generated
                over $k$ from Corollary~\ref{ref:gradedhassamelengthfinite:cor} it follows that
                \[
                h_{\gr \DAA/\DCa}(m) = \sum_{e+1\leq n} \iaq{m, n} =
                \sum_{n=e+1}^{j+1-m}  \iaq{m, n} = \sum_{i=0}^{j+1-m-(e+1)}
                \Dhd{i}{m} = \sum_{i=0}^{a} \Dhd{i}{m}.
                \]
            \item The ideal $\iac{m, j - (m+j)} = \iac{m, -m}$ is zero; thus $I_j = 0$ and $A = A/I_{j}$. Now
                take $a = j$ in Point~\ref{item:subquotients}.\qedhere
        \end{enumerate}
    \end{proof}

    \begin{remark}
        It follows from the proof that $\Dhdvect{j} = (0)$ and $\Dhdvect{j-1}
        = (0, 0)$ for any algebra, so we will ignore these vectors. The
        maximal $a$ such that $\Dhdvect{a}$ may have a non-zero entry is $a =
        j-2$ and indeed $\Dhdvect{j-2} = (0, *, 0)$ contains
        useful information about the algebra, as we will see in Proposition~\ref{ref:squares:prop}.
    \end{remark}

    \begin{example}
        Below we write the $\iaq{m, n}$ table for $j = 3$, which is the smallest
        nontrivial case.

            \def\iainter#1#2{\Dmm^{#1}\cap \Dmmperp{#2}}%
            \def\iainterfst#1{\Dmm^{#1}\cap \Dmmperp{}}%
            \def\iaintersnd#1{\Dmm\cap \Dmmperp{#1}}%
            \def\iainterzero{\Dmm\cap \Dmmperp{}}%
            \begin{center}{%
            \begin{tikzpicture}
                \tikzstyle{every node}=[font=\large]
                \matrix [matrix of math nodes, column
                sep=0 mm, row sep=0mm, ampersand replacement=\&]
                {
                \dimk \frac{A}{\Dmm} \& 0 \& 0  \& \node
                (righttop){0};\\
                0 \& \dimk \frac{\iaintersnd{3}}{\iainter{2}{3} +
                \iaintersnd{2}} \& 0 \& 0\\
                0 \& \dimk \frac{\iaintersnd{2}}{\iainter{2}{2} +
                \iainterzero} \& \dimk \frac{\iainter{2}{2}}{\iainter{3}{2} +
                \iainterfst{2}} \& 0\\
                0 \& 0 \& 0 \& \dimk \iainterfst{3}\\
                };
            \end{tikzpicture}}\end{center}
            Hence the Hilbert function has decomposition
            $\Dhdvect{0} = (1, h(2), h(2), 1),\ \Dhdvect{1} = (0, h(1) - h(2),
            0)$ thus satisfying $h(1) \geq
            h(2)$.
    \end{example}

    \begin{example}
        From Point~\ref{item:subquotients} of Theorem
        \ref{ref:iarrobinoHfdecomposition:thm} it follows that there does not
        exist a local Gorenstein algebra $A$ with Hilbert function decomposition
            \[\begin{matrix}
                \Dhdvect{0}=&(1,& 1,& 1,& 1, & 1, & 1)\\
                \Dhdvect{1}=&(0,& 0,& 1, & 0, & 0)\\
                \Dhdvect{2}=&(0,& 0,& 0, &0)\\
                \Dhdvect{3}=&(0,& 1,& 0).
            \end{matrix}\]
        Indeed $v = \Dhdvect{0} + \Dhdvect{1} = (1, 1, 2, 1, 1, 1)$ should be a
        local Hilbert function of some finite $k$-algebra $B$ and this is
        impossible because $v(1) = 1$ and $v(2) > 1$.
        Note that $h_A = (1, 2, 2, 1, 1, 1)$ seems possible to obtain and
        indeed there exist Gorenstein algebras with such function and
        decomposition
            \[\begin{matrix}
                \Dhdvect{0}=&(1,& 1,& 1,& 1, & 1, & 1)\\
                \Dhdvect{1}=&(0,& 0,& 0, & 0, & 0)\\
                \Dhdvect{2}=&(0,& 1,& 1, &0)\\
                \Dhdvect{3}=&(0,& 0,& 0).
            \end{matrix}\]
    \end{example}

    \subsection{Relations between the Hilbert function and the dual generator}\label{ref:dualgen:section}

    Macaulay's inverse systems and the local Hilbert function decomposition provide us
    with multitude of information about the local finite
    Gorenstein algebra. In this section we will draw various conclusions
    about the structure of the algebra from the knowledge of its dual socle generator. Understanding a Gorenstein algebra $A$ by
    analysis of the dual socle generator $f\in \DSS$ of a chosen presentation
    $A  \simeq \DTT/f^{\perp}$ is elementary. The main problem of this method is embarras de richesses coming from
    the fact that we have many choices of $f$. To avoid it we classify polynomials, provide
    normal forms etc. The most important results in this direction are
    Lemma~\ref{ref:lowdegreesindecomposition:lem} and Theorem
    \ref{ref:standardform:thm}.

    It is important to understand that most complications arise from the fact
    that the dual socle generator is in general not homogeneous. For the
    homogeneous case see
    Lemma~\ref{ref:gradedhavesymmetrichilbertfunction:lem}.

    Fix $\DSS = \DSring$ and $\DTT = \DTring$ as defined in the Theorem
    \ref{ref:Macaulaycorrcor:thm}. By $\DSS^{m}$ we denote the homogeneous
    polynomials of total degree $m$ is $\DSS$. Fix an element $f\in \DSS$ and a local finite Gorenstein
    $k$-algebra $(A, \Dmm, k)$ isomorphic to $\DTT/f^{\perp}$. Let $j$ be the
    socle degree of $A$, it is equal to the degree of the polynomial $f$.
    It is important to know that, although using Theorem
    \ref{ref:Macaulaycorrcor:thm} we are bound to consider only
    characteristic zero, the facts from this section remain true if we replace
    $\DTring$ by the divided powers ring and use Theorem
    \ref{ref:Macaulaycorr:thm}.%
    \def\iatf#1#2{(\DTT f)_{#1}^{#2}}%

        Define $\iatf{n}{}$ to be the submodule of $\DTT f$ spanned by elements
        of degree less than $n$ and $\iatf{}{m} := \Dmm_{\DTT}^{m}f$, finally
        put $\iatf{n}{m} := \iatf{n}{}\cap \iatf{}{m}$.

    \begin{prop}\label{ref:hilbertfunctionfrompolynomial:prop}
        \def\iaq#1{Q_{#1}}
        We have an isomorphism of $\DTT$-modules
        \[
        i:A \to \DTT f \subseteq \DSS,\mbox{ defined by } \partial\mapsto
        \partial\hook f.
        \]
        The submodule $\iac{m, n}$, defined in Theorem~\ref{ref:iarrobinoHfdecomposition:thm}, maps to $\iatf{n}{m}$ under this
        isomorphism and thus the $k$-rank of $\iaq{m, n}$, also defined in
        Theorem~\ref{ref:iarrobinoHfdecomposition:thm}, is equal to
        \[
        \dimk \frac{\iatf{n}{m}}{\iatf{n-1}{m} + \iatf{n}{m+1}} = \dimk
        \iatf{n}{m} - \dimk \left(  \iatf{n-1}{m} + \iatf{n}{m+1}\right).
        \]
        Furthermore, the value of the Hilbert function $h_A(t)$ is equal to
        both $\dimk \iatf{}{t} - \dimk \iatf{}{t+1}$ and $\dimk \iatf{t+1}{} -
        \dimk
        \iatf{t}{}$.
    \end{prop}
    \begin{proof}
        The claim on $\iac{m, n}$ follows directly from definitions. For the
        part concerned with the Hilbert function check Corollary~\ref{ref:loewyHilbertfunc:cor}. 
    \end{proof}

    \begin{example}
        Using Proposition~\ref{ref:hilbertfunctionfrompolynomial:prop} one can
        compute the Hilbert function of the apolar algebra $A$ of $x_1^5 +
        x_2^4 + x_3^4$.
        We have $\iatf{1}{} = k, \iatf{2}{} = k \oplus kx_1\oplus kx_2\oplus
        kx_3$ and so on, hence we compute $h_A = (1, 3, 3, 3, 1, 1)$.
    \end{example}

    It is important to see that computing $\dimk \iac{m, n}$ or $h_A$ via
    Proposition~\ref{ref:hilbertfunctionfrompolynomial:prop} we are mainly interested in the
    top degree forms of the derivatives.

    The situation is particularly easy when the dual socle generator is
    homogeneous.
    \begin{lem}\label{ref:gradedhavesymmetrichilbertfunction:lem}
        If $f$ is homogeneous then $\Dhd{m}{n} = 0$ for all $m > 0$ and $n$.
        In particular the local Hilbert function of $A$ is
        equal to $\Dhdvect{0}$.
    \end{lem}

    \begin{proof}
        Since for $m, n$ such that $m + n < j+1$ we have  $\iatf{n}{m} = 0$
        the claim follows from
        Proposition~\ref{ref:hilbertfunctionfrompolynomial:prop}.
%        As $f$ is homogeneous we can choose a splitting and projection $\pi$
%        in Proposition~\ref{ref:ppairingonGorenstein:prop} in a homogeneous
%        way and deduce that $\Dmmperp{n} = \Dmm^{j+1-n}$, then the claim
%        follows from definition of $\Dhd{m}{n}$.
    \end{proof}

    \begin{remark}\label{ref:baddecompositions:remark}
        One might hope that, for every $a\geq 0$, if $f$ is a sum of polynomials of degrees at least
        $j-a$,
        then $\Dhdvect{b}$ has zero entries for all $b > a$. But this is not true.
        Take $f = x_1^4 - 12x_1^2x_2$, then $(y_1^2 + y_2)\hook f = -24x_2$
        and one may check that this is a non-zero element of
        $\iatf{2}{1}/\pp{\iatf{1}{1} + \iatf{2}{2}}$. In
        particular $\Dhdvect{2} \neq (0, 0, 0)$ and in fact the decomposition
        is
        \[\begin{matrix}
            \Dhdvect{0}=&(1,& 1,& 1,& 1, & 1)\\
            \Dhdvect{1}=&(0,& 0,& 0, & 0)\\
            \Dhdvect{2}=&(0,& 1,& 0).
        \end{matrix}\]
    \end{remark}

    Intuitively, low degree homogeneous terms of $f$ contribute
    only to $\iatf{n}{m}$ for which $m + n$ is small. The following
    lemma captures this intuition and says ``if you add low degree term, you
    change only $\Dhdvect{s}$ for $s\gg 0$''.
    \begin{lem}\label{ref:lowdegreesindecomposition:lem}
        \def\Dhdprim#1#2{\Delta_{#1}'\pp{#2}}
        Suppose that polynomials $f_1, f_2\in \DSS$ of degree $j$ are such that
        $\deg (f_1 - f_2) \leq j - a$. Denote $\Delta =
        \Delta_{\DTT/f_1^{\perp}}$ and $\Delta' = \Delta_{\DTT/f_2^{\perp}}$, then
        \[\Dhd{m}{n} = \Dhdprim{m}{n}\mbox{
        for all } m\leq a-1\mbox{ and all } n.\]
    \end{lem}

    \begin{proof}
        \def\iatfone#1#2{(\DTT f_1)_{#1}^{#2}}
        \def\iatftwo#1#2{(\DTT f_2)_{#1}^{#2}}
        In the computation of $\Dhd{m}{n}$, where $m\leq a -1$, only $\iaq{m,
        n}$ such that $(j+1) - (m + n) \leq a - 1$ are used. By Proposition
        \ref{ref:hilbertfunctionfrompolynomial:prop} $k$-ranks of such
        $\iaq{m, n}$ depend only on $k$-ranks of $\iatfone{n}{m}$ such that
        $(j+1) - (m+n) \leq a$.
        Let us fix $m$ and take two canonical epimorphisms
        \[\pi_i: \DD{mmT}{\Dmm_{\DTT}}^m\to \DTT f_i,\quad \pi_i(\partial) =
        \partial\hook f_i\quad\mbox{for}\quad i=1,2.\]
        We claim that if $(j+1) - (m+n) \leq a$ then the preimages $\pi_1^{-1}\iatfone{n}{m},
        \pi_2^{-1}\iatftwo{n}{m}$ are equal. First, if $\partial\in \DmmT^m$,
        then $\deg \partial\hook (f_1 - f_2) \leq j - a - m$.
        An element
        $\partial\in \DmmT^m$ satisfies $\pi_1(\partial) \in \iatfone{n}{m}$
        if and only if $\deg \partial \hook f_1 < n$. Similarly $\pi_2(\partial)\in
        \iatftwo{n}{m}$ if and only if $\deg\partial \hook f_2 < n$.
        But $n \geq j+1-a-m > \deg\partial\hook (f_1 - f_2)$ so both
        conditions are equivalent.

        Now
        \def\preimageone#1#2{\pi_1^{-1}\pp{\iatfone{#1}{#2}}}%
        \def\preimagetwo#1#2{\pi_2^{-1}\pp{\iatftwo{#1}{#2}}}%
        \begin{equation}\label{item:quotientsareok}
            \frac{\iatfone{n}{m}}{\iatfone{n-1}{m}}  \simeq
            \frac{\preimageone{n}{m}}{\preimageone{n-1}{m}} =
            \frac{\preimagetwo{n}{m}}{\preimagetwo{n-1}{m}}  \simeq
            \frac{\iatftwo{n}{m}}{\iatftwo{n-1}{m}}.
        \end{equation}
        Since $\iatf{n-1}{m}\cap \iatf{n}{m+1} = \iatf{n-1}{m+1}$ we have
        inclusions
        \[
        \frac{\iatfone{n}{m+1}}{\iatfone{n-1}{m+1}} \into
        \frac{\iatfone{n}{m}}{\iatfone{n-1}{m}}\quad\mbox{and}\quad
        \frac{\iatftwo{n}{m+1}}{\iatftwo{n-1}{m+1}} \into
        \frac{\iatftwo{n}{m}}{\iatftwo{n-1}{m}},
        \]
        thus the claim follows from isomorphisms (\ref{item:quotientsareok}).
    \end{proof}

        Example~\ref{ref:differentdecompositions:example} shows that the inequality
        $m\leq a - 1$ is strict.

    \begin{cor}\label{ref:topdegreefirstrow:cor}
        \begin{enumerate}
            \item The sequence $\Dhdvect{\DTT/f^{\perp}, 0}$ is the Hilbert function
        of the top degree form of~$f$.
            \item If the Hilbert function of $\DTT/f^{\perp}$ is
                symmetric~i.e.~$h(t) = h(j-t)$ for all $0\leq t\leq j$, then
                it is equal to $\Dhdvect{0}$.
        \end{enumerate}
    \end{cor}

    \begin{proof}
        \begin{enumerate}
            \item Let $\DD{tmpftop}{f_{top}}$ be the top degree form of $f$,
                then $\deg f - \Dtmpftop\leq j -1$ and from Lemma~\ref{ref:lowdegreesindecomposition:lem} it
                follows that $\Dhdvect{0}$ are equal for $f$ and $f_{top}$.
                But $f_{top}$ is homogeneous so from Lemma
                \ref{ref:gradedhavesymmetrichilbertfunction:lem} it follows
                that $\Dhdvect{0}$ is the Hilbert function of $\Dtmpftop$.
            \item The proof is purely combinatorial --- the only important fact
                is that $\Dhdvect{i}$ are symmetric with
                respect to $\frac{j-i}{2}$ for $i\geq 0$, this follows from
                Theorem~\ref{ref:iarrobinoHfdecomposition:thm}.\qedhere
        \end{enumerate}
    \end{proof}

    \def\Dy{y}
    \begin{example}\label{ref:differentdecompositions:example}
            Let us take $n = 2,\ \DTT = k\formal{\Dy_1, \Dy_2}$ as defined in
            Theorem~\ref{ref:Macaulaycorrcor:thm} and an apolar algebra $A = \DTT/(x_1^3 -
            x_2^2)^{\perp}$.
            Since $\Dhdvect{0}$
            is the Hilbert function of $\DTT/(x_1^3)^{\perp}$ is it equal to $(1,
            1, 1, 1)$. Moreover as $h_A(1) = 2$ and $\Dhdvect{1} = (0, a, 0)$
            we see that $1 + a = 2$, so that $a = 1$ and the full
            decomposition is
            \[\begin{matrix}
                \Dhdvect{0}=&(1,& 1,& 1,& 1)\\
                \Dhdvect{1}=&(0,& 1,& 0)
            \end{matrix}\]
            yielding $\dimk A = 5$ and $h_A = (1, 2, 1, 1)$.

            On the other hand let us take $B = \DTT/(x_1^3 - x_1\cdot
            x_2)^{\perp},\ f:= x_1^3 - x_1 x_2$. Then one calculates that $\iatf{}{1} = k\oplus
            kx_1$, $x_1 x_2$ is ``shadowed by $x_1^3$'' and so the Hilbert function decomposition
            is
            \[\begin{matrix}
                \Dhdvect{0}=&(1,& 1,& 1,& 1)\\
                \Dhdvect{1}=&(0,& 0,& 0).
            \end{matrix}\]
            We will now see that this difficulty can be overrun by
            standardising $f$. Unfortunately this method works only for
            problems occurring in low degrees, in general the problem of
            removing ``exotic summands''
            is difficult, see \cite{BMR}.
    \end{example}

    Theorem~\ref{ref:Macaulaycorrcor:thm} presents any finite
    local Gorenstein $k$-algebra as $\DTT/f^{\perp}$, but the choice of $f$ is not
    unique. However we have particularly good choices of $f$,
    each of which is named the standard form of $f$.

    \begin{defn}
        For a finite Gorenstein $k$-algebra $A  \simeq \DTT/I$ of socle degree
        $j$ let $\Dhdvect{\bullet}$ be the
        decomposition of the Hilbert function of $A$ and $e(i) := \sum_{t=0}^i
        \Dhd{t}{1}$.
        If $f\in \DSS$ is a dual socle generator of $\DTT/I$ then
        $f$ \emph{is in the standard form} iff
        \[
        f = f_0 + f_1 + f_2 + f_3 + \dots + f_j\mbox{ for } f_i\in \DSS^i \cap
        k[x_1,\dots,x_{e(j-i)}]
        \]
        and either the socle degree of $A$ is at most one or $f_0 = f_1 = 0$.
    \end{defn}
    
    The following theorem proves that a standard form exists for any $f\in
    \DSS$.

    \begin{thm}[Standard form of a dual socle generator]\label{ref:standardform:thm}
        Let $(A, \Dmm, k)$ be a local finite Gorenstein $k$-algebra of socle
        degree $j$
        with $h_A(1) = n$. Set $\DSS = \DSring$ and $\DTT = \DTring$.
        Then there exists $f = f_0 + \dots + f_j \in \DSS$, such that
        \begin{enumerate}
            \item $\DTT/f^{\perp}  \simeq A$,
            \item $f_i\in S^i \cap k[x_1,\dots, x_{e(j-i)}]$, where $e(i) =
                \sum_{t=0}^i \Dhd{t}{1}$.
        \end{enumerate}
        If $j\geq 2$ then one can furthermore assume $f_0 = f_1 = 0$.
    \end{thm}

%        \[\Dmm = \iaintersnd{j} \supseteq \iaintersnd{j-1} \supseteq \dots
%            \supseteq \Dmm \cap \Dmmperp{} \supseteq 0.\]
    \begin{proof}
        Choose first any $f' \in \DSS$ such that
        $\DTT/(f')^{\perp} \simeq A$; this is possible by Theorem
        \ref{ref:Macaulaycorrcor:thm}. Denote $I := (f')^{\perp}$. Write $f' = f'_0 + \dots + f'_j$, where
        $f'_i\in S^i$ for $i=0,\dots j$.
        Consider the sequence of ideals
        \def\iaintersnd#1{\Dmm\cap \Dmmperp{#1}}%
        \def\fracced#1{\frac{#1 + \Dmm^2}{\Dmm^2}}
        \begin{equation}\label{eq:standardform}\frac{\Dmm^{\phantom{2}}}{\Dmm^2} =
            \fracced{\iaintersnd{j}} \supseteq
            \fracced{\iaintersnd{j-1}} \supseteq \dots
            \supseteq \fracced{\Dmm \cap \Dmmperp{}} \supseteq
            0,\end{equation}
            then from 
            Theorem~\ref{ref:iarrobinoHfdecomposition:thm}
            Point~\ref{item:subquotients} it follows that
        \begin{align*}
            \dimk \frac{\iaintersnd{j-a-1} + \Dmm^2}{\Dmm^2} = \dimk
            \frac{\Dmm^{\phantom{2}}}{\Dmm^2} - e(a),
        \end{align*}
        for all $-1\leq a\leq j-1$, where $e(-1) := 0$.
        Now we choose lifings of the bases of the $\DTT$-modules from
        Equation \eqref{eq:standardform}. More precisely take $z_n,
        z_{n-1},\dots, z_1\in\DmmT \setminus \DmmT^2$
        such that for all $-1\leq a\leq j-1$ the images of $z_n,
        z_{n-1},\dots, z_{e(a) + 1}$ in $\DTT/I  \simeq A$ lie
        in $\iaintersnd{j-a-1}$ and their images in $A/\Dmm^2$ form a $k$-basis of
        $\frac{\iaintersnd{j-a-1} + \Dmm^2}{\Dmm^2}$.
        In particular the case $a = -1$ implies that the images of $z_n, \dots, z_1$ form a
        $k$-basis of $\Dmm/\Dmm^2 \simeq \DmmT/\DmmT^2$, so $k\formal{z_1,\dots,z_n} = \DTT$.

        For any $r$ and $a$ such that $r > e(a)$ the image of $z_r$ in
        $A$ lies in $\Dmmperp{j-a-1}$ so that the image of $\DmmT^{j-a-1}z_r$
        in $A$ is zero, in other words
        \begin{equation}\label{eq:stformzeroing}\DmmT^{j-a-1}z_r \subseteq
        I.\end{equation}

        Consider the automorphism $\varphi$ of $\DTT$ sending $z_i$ to $y_i$, let the
        image of $I$ under this automorphism be equal to the annihilator of
        some polynomial $f\in \DSS$.
        Note that for any $r$ and $i$ such that $r > e(j-i)$ by Equation
        \eqref{eq:stformzeroing} we have
        \begin{equation}\label{eq:stformvarzero}
            \DmmT^{i-1}y_r = \DmmT^{j-(j-i)-1} y_r =
            \varphi\pp{\DmmT^{j-(j-i)-1}z_r} \subseteq \varphi(I) = (f)^{\perp}.
        \end{equation}
        Write $f = f_j + f_{j-1} + \dots + f_0$, where $f_i\in S^i$.
        We claim that
        \[f_i\in \DSS^i \cap k[x_1,\dots, x_{e(j-i)}].\]
        We prove it by downward induction on $i= j, j-1,\dots, 0$.
        Let us present the induction step. Take any $i$ such that $0 \leq i\leq j$
        and suppose that $f_d \in \DSS^{d} \cap k[x_1,\dots, x_{e(j-d)}]$ for
        all $d > i$. Take any $r > e(j-i)$, then $\DmmT^{i-1}y_r\hook f = 0$
        by Equation \eqref{eq:stformzeroing}.
        By induction hypothesis $y_r\hook f_d = 0$ for all $d > i$. Moreover
        $\DmmT^{i-1}y_r \subseteq \DmmT^{i}$ so that $\DmmT^{i-1}y_r \hook
        f_d = 0$ for $d<i$. This means that
        \[
        0 = \DmmT^{i-1}y_r \hook f = \DmmT^{i-1}y_r \hook f_i =
        \DmmT^{i-1}\pp{y_r \hook f_i}.
        \]
        But $f_i$ is homogeneous of degree $i$
        so $y_r \hook f_i$ is homogeneous of degree $i-1$ and
        annihilated by $\DmmT^{i-1}$, thus $y_r\hook f_i = 0$. This proves that
        \[f_i\in S^i \cap k[x_1,\dots, x_{e(j-i)}].\]
        The basis of the induction is proved in the same way (with no $d >
        i$).

        Now suppose $j\geq 2$. Since $h_A(1) = n$, by
        Proposition~\ref{ref:hilbertfunctionfrompolynomial:prop} we see that any
        polynomial of degree at most one is a derivative of $f$. In particular
        $f_0 + f_1 = \partial\hook f$ for some $\partial\in \DTT$. Since $j\geq
        2$ we have $\partial\in \DmmT$, then $f - (f_0 + f_1) = (1 -
        \partial)\hook f$ generates the same $\DTT$ submodule of $\DSS$ as
        $f$ and
        $f^{\perp} = \pp{f - f_0 - f_1}^{\perp}$ follows.
    \end{proof}

    \subsection{Families of dual generators and morphisms to Hilbert scheme}

    In this subsection we obtain tools for the reasoning ``since the set of dual
    generators is irreducible, the corresponding apolar algebras form an
    irreducible subset of the Hilbert scheme''. But the idea is deeper
    --- we obtain morphisms from affine varieties to the Hilbert scheme,
    which give information about its geometry.

    \begin{prop}\label{ref:familiesofpolystohilb:prop}
        \def\DPn{\mathbb{P}^n}
        Fix positive integers $j, r$ and a point $p\in \DPn$ with an affine
        neighbourhood $U$.
        View $\DD{Sleqj}{\DSS^{\leq j}}$ as an affine space and let $V \subseteq
        \DSleqj$
        be a Zariski-constructible subset such that for every closed point $f\in V$ the
        apolar algebra of $f$ has rank $r$. Then $V$ induces a morphism
        \[\varphi:V\to \DD{HilbGPn}{\operatorname{HilbGor}_{\DPn}^{r}}\]
        such that for every closed point $f\in V$ the image $\varphi(f)\in
        \DHilbGPn$ corresponds to a closed subscheme $W_f\subseteq \DPn$ supported
        at $p$ and with $\Gamma(U, W_f)$ isomorphic to the apolar algebra of
        $f$.
    \end{prop}

    \begin{proof}
        \def\DPn{\mathbb{P}^n}
        We will construct a closed subscheme $Z \subseteq \DPn \times V$ flat
        over $V$ and with suitable fibers, then the required morphism will
        follow from
        the universal property of the Hilbert scheme. We can view
        $R := \Spec \pp{\DD{Trealleqj}{\DTT/\DmmT^{j+1}}}$ as a closed subscheme of $\DPn$ supported at
        $p$ and we will in fact construct $Z \subseteq R
        \times V$.

        We first treat $\DTrealleqj$ formally. Let $W \to\DTrealleqj$ be an
        isomorphism of $k$-vector spaces. Consider the vector bundle
        \def\Sym{\operatorname{Sym}}
        $\DD{lW}{\mc{W}} = \Spec \pp{\Sym W^*} \times V$ over $V$. Define a subvariety
        $\DD{lZ}{\mathcal{A}nn}
        \subseteq \DlW$ by
        \[
        \DlZ = \left\{ (\partial, f)\in \DlW\ |\ \partial\hook f = 0 \right\},
        \]
        then the fibers of $\DlZ$ over closed points of $V$ are affine spaces
        of constant
        dimension, so $\DlZ$ is a subbundle of $\DlW$.
        Let 
        $\DD{lQ}{\mc{Q}}$ be the quotient bundle, then dualizing we have an inclusion of vector bundles $\DD{lQstar}{\mc{Q}^*}
        \subseteq \DD{lWstar}{\mc{W}^*}$.
        
        The bundle $\DlWstar$ is canonically
        isomorphic to $\Spec \pp{\Sym W} \times V$. The $k$-linear isomorphism $W\to
        \DTrealleqj$ induces morphisms
        $R \to \Spec \pp{\Sym W}$ and $R \times V \to \DlWstar$.
        We claim that the pullback
        \[\DD{tmpZ}{Z} := \pp{R \times V}\times_{\DlWstar}
        \DlQstar \subseteq R \times V\]
        has the properties required.
        Choose a closed point $f\in V$. The fiber of $\DlQ$ over
        $f$ is the affine space $\DlW_f/\DlZ_f$ and the fiber of $\DlQstar$ over $f$ is
        the affine space $\Homt{k}{\DlW_f/\DlZ_f}{k}$, which is defined by the ideal
        of $\Sym W$ generated by the linear subspace
        $\annn{R}{f} \subseteq
        W$. It follows that the fiber of $\DtmpZ$ over $f$ is isomorphic to
        $\Spec\pp{R/\annn{R}{f}}$, the affine scheme of the apolar algebra of
        $f$.
        The scheme $Z$ is finite over $V$ and its fibers over closed
        points of $V$ have rank $r$. From \cite[II.5 Ex
        5.8bc]{HarAG} it follows that $\OO{Z}$ is a locally free $V$ module,
        thus in particular $Z$ is flat over $V$.
    \end{proof}

    \begin{prop}\label{ref:maximalopen:prop}
        Let $j$ be a positive integer and $V \subseteq \DSS^{\leq j}$ be a Zariski-constructible subset. Then 
        the set of $f\in V$ such that the apolar algebra of $f$ has maximal
        rank over $k$ is open in $V$.
    \end{prop}

    \begin{proof}
        This is a direct consequence of semi-continuity of dimension of fibers
        and the proof of Proposition~\ref{ref:familiesofpolystohilb:prop}.
    \end{proof}

    \section{Applications}\label{ref:originalresearch:sec}

    In this section for simplicity we assume that $k = \mathbb{C}$, though
    many of the results may be proved with much weaker assumptions.
    Our objective is to present applications of the above theory,
    in particular we will be interested in finding deformations and
    smoothings of zero-dimensional Gorenstein subschemes of
    $\DD{Pn}{\mathbb{P}^n} = \DD{Pnk}{\mathbb{P}_k^n}$.
    Supported by Proposition
    \ref{ref:smoothabilityeverywhere:prop} we will investigate
    $\DD{HilbGPn}{\operatorname{HilbGor}_{\DPn}^{r}}$.
    The picture here was mainly drawn by Iarrobino \cite{ia94}, who proved
    that ``the majority'' of algebras of large enough rank is not smoothable. In
    particular he showed an example of non-smoothable Gorenstein algebra of
    rank $14$ by computing that the tangent space rank of the corresponding point
    of the Hilbert scheme of $\mathbb{P}^6$ is less than $6\cdot 14$, which is
    the dimension of the smoothable component. On the other hand Casnati and
    Notari (see~\cite{CN10}) proved that algebras of rank at most $10$
    are smoothable. It would be interesting to know what is the minimal
    rank of a non-smoothable Gorenstein $k$-algebra, we hope to answer this
    question in a joint paper with Gf. Casnati and R. Notari. In this section
    we content ourselves with analyzing three examples.
    
    To prove that a subset $V \subseteq
    \DHilbGPn$ lies in
    $\DD{HilbGPnzero}{{\operatorname{HilbGor}_{\DPn}^{r}}^{\circ}} :=
    \DHilbGPn \cap {\Hilb_{\DPn}^{r}}^{\circ}$
    we will apply ternary approaches:
    \begin{enumerate}
        \item Prove directly that for any $[R]\in V$ the scheme $R\subseteq
            \mathbb{P}^n$ is smoothable;
        \item Prove that $V$ is irreducible and find a smooth point of the
            Hilbert scheme contained in $V\cap \DHilbGPnzero$;
        \item Prove that $V$ is irreducible and that a general point of $V$
            lies in $\DHilbGPnzero$.
    \end{enumerate}

    We will use the first approach to prove smoothability of Gorenstein
    algebras with Hilbert function $(1, 5, 4, 1)$, the second approach to prove
    smoothability of Gorenstein algebras having
    Hilbert function $(1, 4, 4, 3, 1)$ and the third approach to prove
    smoothability of \emph{graded} Gorenstein algebras with Hilbert function
    $(1, 3, 3, 3, 3, 1)$.

    The most important facts that we will use without further reference are
    Proposition~\ref{ref:smoothabilityeverywhere:prop}, which allows us not
    to care too much about the ambient space and the Definition
    \ref{ref:defsmoothable:def} of the
    smoothable component, which implies that to prove smoothability of $X$ it is
    sufficient to show a deformation with a fiber $X$ over a closed point and
    a general fiber \emph{smoothable}, not necessarily smooth.

    Before proving our main results we provide the technical background by
    further analyzing the dual socle generators and apolar ideals in the
    general situation covering the above
    cases.

    \subsection{Enhancements of Iarrobino symmetric decompositions}

    The following proposition further standardises the dual
    generator. It is used in proof of Theorem~\ref{ref:bjorktactical:thm}.
    \begin{prop}\label{ref:squares:prop}
        Let $f\in \DSS$ be a polynomial of degree $j\geq 2$ such that the Hilbert function decomposition from
        Theorem~\ref{ref:iarrobinoHfdecomposition:thm} has $\Dhdvect{j-2} =
        (0, q, 0)$. Then the apolar algebra of $f$ is isomorphic to the apolar
        algebra of $g\in \DSS$, where $g = g_j + g_{j-1} + \dots + g_2$ is the standard form from
        Theorem~\ref{ref:standardform:thm} and furthermore $g_2$ is a sum of
        $q$ squares of variables not appearing in $g_{\geq 3}$ and a quadric
        in variables appearing in $g_{\geq 3}$.
    \end{prop}

    \begin{proof}
        Applying Theorem~\ref{ref:standardform:thm} we obtain $g$ in the
        required form except from, perhaps, the assumptions on $g_2$.
        We will prove the theorem when $j\geq 3$, the case $j = 2$ is easy and
        we leave it to the reader.
        Let $f:=e(j-3) = \sum_{t=0}^{j-3} \Dhd{t}{1},\ e:=  f + q$, then $g_{\geq 3}\in k[x_1,\dots,x_{f}]$ and $g_2\in
        k[x_1,\dots,x_e]$.
        We can diagonalize $g_2$ with respect to $x_{f+1},\dots, x_e$ obtaining
        \[
        g_2 = \sum_{i=f+1}^e \lambda_i\cdot x_i^2 + Q
        \]
        where $Q\in (x_1,\dots,x_f)$ and $\lambda_i\in \{0, 1\}$.
        Note that after this operation $g$ is still in the standard form.
        If all
        $\lambda_{\bullet}$ are equal to $1$, then by another linear change of
        coordinates $x_{f+1},\dots,x_{e}$ we obtain $g_2 = \sum_{i=f+1}^e
        x_i'^2 + Q'$ where $Q'\in k[x_1,\dots,x_f]$, thus the claim follows.
        If for some $i$ we have $\lambda_i = 0$ then
        $y_i^*\hook g \in k[x_1,\dots,x_f]$, thus from a $k$-rank
        count it follows that $\dimk \iatf{2}{} g < n+1$ but this contradicts the fact that $\dimk
        \iatf{2}{} g = h(0) + h(1) = n + 1$, see Proposition
        \ref{ref:hilbertfunctionfrompolynomial:prop}.
    \end{proof}

    \subsection{Application of flatness criterion from
    Theorem~\ref{ref:bjorkflatness:thm}}

    There exist polynomials $f\in S$ such that there is an easily described
    flat family with general member reducible and a special member
    isomorphic to the apolar algebra of $f$. We present them below.
    
    \def\Dx{X}
    \begin{prop}\label{ref:addingvartrick:prop}
        \def\parf{\partial\hook f}
        \def\dual#1{#1^{*}}
        Let $\DSS = \DSring$ and $T = S[\Dx]$. Let $f\in S$ and
        $\partial\in \dual{S}$ be such that $\partial^2\hook f = 0$. Take a
        natural number $m \geq 2$ and set
        $g := f + \Dx^{m}\cdot \parf$, then
        \[
        \annn{\dual{T}}{g} = \dual{T}\cdot \annn{\dual{S}}{f} + \dual{T}\cdot \alpha\cdot
        \annn{\dual{S}}{\parf} + \left( \alpha^m - m!\cdot\partial \right),
        \]
        where $\alpha\in \dual{T}$ is the dual to $\Dx$.
    \end{prop}

    \begin{proof}
        Clearly the element $\alpha^{m} - m!\cdot \partial$ annihilates $g$,
        and reducing modulo it, we may investigate only elements of the form
        $\sigma_0 + \sigma_1\cdot \alpha+ \dots +
        \sigma_{m-1}\alpha^{m-1}$, where $\sigma_i\in S^*$. The action of such
        element on $g$ is given by
        \begin{multline*}
        \left(\sigma_0 + \sigma_1\cdot \alpha+ \dots +
        \sigma_{m-1}\alpha^{m-1}\right)\hook g =\\\sigma_0\hook f +
        \Dx^{m}(\sigma_0\partial\hook f) + m\cdot\Dx^{m-1}(\sigma_1\partial\hook f) + \dots
        + m!\cdot\Dx(\sigma_{m-1}\partial\hook f),
        \end{multline*}
        thus the element belongs to $\annn{T^*}{g}$ if and only if $\sigma_0 \hook f = 0,
        \sigma_1\partial\hook f = 0,\dots,\sigma_{m-1}\partial\hook f =0$,
        and the claim follows.
    \end{proof}

    In the following theorem we use Theorem~\ref{ref:bjorkflatness:thm} in an
    essential way -- without any knowledge of the fibers, we prove that a
    certain family is flat. This theorem will act as the main ingredient of almost
    all proofs of smoothability, allowing one to reduce the question to the
    smoothability of algebras of lower $k$-rank.
    \begin{thm}\label{ref:bjorktactical:thm}
        Let $T^* = k\formal{\Dy_1,\dots,\Dy_n, \alpha}$ and $A = \DD{tmpT}{T^*}/I$ be a finite
        $k$-algebra defined by the ideal
        \[
        I = (\alpha^{o} - q) + J,
        \]
        where $o\geq 2, q\in k[\Dy_1,\dots,\Dy_n]$ and $J\ideal \DtmpT$ is generated by
        elements of the ideal $(\Dy_1,\dots,\Dy_n)\ideal k[\Dy_1,\dots,\Dy_n, \alpha]$ homogeneous with
        respect to the grading by $\alpha$.
        Suppose that for some natural $c$ such that $0< c < o$ we have
        \[\DD{tmpan}{\annn{\DtmpT/J}{\DD{yo}{\alpha^o}}} \subseteq
        \annn{\DtmpT/J}{\DD{yom}{\alpha^{c}}}\mbox{ and } \Dtmpan \subseteq
        \annn{\DtmpT/J}{q},\]
        then $\Spec A$ is a degeneration of reducible schemes. More precisely
        the morphism
        \[
        \varphi:\Spec \frac{\DD{Tpolyt}{k[\Dy_1,\dots,\Dy_n, \alpha, t]}}{\pp{\Dyo - t\cdot \Dyom -
        q} + (J)} \to \Spec \DD{line}{k[t]}
        \]
        is flat, with a general fiber reducible and the fiber over $t$
        isomorphic to $\Spec A$.
    \end{thm}

    \begin{proof}
        Denote by $I_t = \pp{\Dyo - t\cdot \Dyom -
        q} + (J)\ideal \DTpolyt$.
        For an invertible $\lambda\in k$, the fiber of $\varphi$ over
        $(t-\lambda)$ is supported at the origin and at
        $\big(0, 0, \dots, 0, \sqrt[o-c]{\lambda}\big)$ thus it is a reducible scheme. It remains to
        prove that $\varphi$ is flat. We will use Theorem
        \ref{ref:bjorkflatness:thm}.
        \def\DTpoly{k[\Dy_1,\dots,\Dy_n,\alpha]}%
        Define a filtration on $\DTpolyt$ by
        gradation with respect to $\alpha$:
        \[
        \DTpolyt_m = \{f\in \DTpolyt\ |\ \deg_{\alpha} f \leq m\}.
        \]
        Since $J$ is homogeneous with respect to the gradation by $\alpha$ we can calculate $\gr I_t$ as the preimage of
        $\gr \pp{\Dyo - t\cdot \Dyom - q} \subseteq \gr \DTpolyt/(J)$. For
        clarity let us denote $R := \DTpolyt/(J)$.
        Now, since $(J)$ is generated by elements of $\DTpoly$, we have
        \[\annn{R}{\Dyo} = k[t]\cdot \pp{\annn{\DTpoly/J}{\Dyo}},\mbox{ so }
        \annn{R}{\Dyo} \subseteq \annn{R}{\Dyo-t\cdot\Dyom-q}
        \]
        and we use Proposition
        \ref{ref:initialideals:prop} to deduce that
        $\gr \pp{\Dyo - t\cdot \Dyom - q} \subseteq \gr R$ is
        generated by $\Dyo$. Now the claim follows from Proposition
        \ref{ref:exactoffiltered:prop} and Theorem
        \ref{ref:bjorkflatness:thm}.
    \end{proof}

    \begin{cor}\label{ref:bjoerktactcor:cor}
        \def\parf{\partial\hook f}
        \def\dual#1{#1^{*}}
        Let $\DSS = \DSring$ and $T = S[\Dx]$. Let $f\in S$ and
        $\partial\in \dual{S}$ be such that $\partial^2\hook f = 0$. Set
        $g = f + \Dx^{m}\cdot \parf$, where $m \geq 2$ is a natural number. The apolar algebra of $g$ is a
        fiber of a deformation over $\Spec\Dline$, whose general fiber is
        isomorphic to a disjoint sum of the apolar
        algebra of $f$ and $m-1$ copies of the apolar algebra of $\parf$.
    \end{cor}

    \begin{proof}
        \def\dual#1{#1^{*}}
        Let, $S^* = \DTring$ and $T^* = k\formal{\Dy_1,\dots,\Dy_n, \alpha}$.
        Clearly $\partial$ may be taken to be a polynomial in
        $\Dy_1,\dots,\Dy_n$.
        From Proposition~\ref{ref:addingvartrick:prop} it follows that
        \[\annn{\dual{T}}{g} = \dual{T}\cdot \annn{\dual{S}}{f} + \dual{T}\cdot \alpha\cdot
        \annn{\dual{S}}{\partial \hook f} + \left( \alpha^m - m!\cdot\partial
        \right),\]
        Let $J = \dual{T}\cdot \annn{\dual{S}}{f} + \dual{T}\cdot \alpha\cdot
        \annn{\dual{S}}{\partial \hook f}$, then
        \[\DD{tmpan}{\annn{\DtmpT/J}{\DD{yo}{\alpha^m}}} =
        \annn{\DtmpT/J}{\DD{yom}{\alpha}}= \annn{\DtmpT/J}{\partial} =
        \pp{\annn{\dual{S}}{\partial\hook f}}/J.\]
        It follows that $m!\cdot\partial, m, 1, J$ satisfy the assumptions of Theorem~\ref{ref:bjorktactical:thm} for $q, o,
        c, J$ respectively. The required deformation $X\to
        \Spec \Dline$ comes from this theorem.
        We will now check that
        the fiber of $X$ over $(t-\lambda)$, where $0\neq \lambda\in k$, is
        isomorphic to a disjoint sum of the apolar algebra to $f$ and $m-1$
        copies of the apolar algebra of $\partial \hook f$.
        This fiber is isomorphic to
        \def\DTpoly{k[\Dy_1,\dots,\Dy_n,\alpha]}%
        \[
        \Spec \frac{\DTpoly}{\DD{tmpspec}{\pp{\alpha^m - \lambda\cdot \alpha -
        m!\cdot\partial}} +
        J},
        \]
        and has support equal to the union of $(0,\dots,0)$ and $(0,\dots,0,
        \omega)$, where $\omega$ runs through the $(m-1)$'st roots
        of $\lambda$. Note that since $\alpha\cdot \partial\in J$ the element
        $\alpha^{m+1} - \lambda\cdot \alpha^2 = \alpha\cdot \pp{\Dtmpspec} +
        m!\cdot\alpha\cdot
        \partial$ is in the ideal $\DD{Il}{I(\lambda)}$ defining the fiber.

        We will now look near $(0,\dots, 0, 0)$, i.e. localise the fiber at the
        ideal $\Dmm = (\Dy_1,\dots,\Dy_n, \alpha)$. In this localisation $\alpha^{m-1} -
        \lambda$
        is invertible, thus $\alpha^2 = \pp{\alpha^{m+1} - \lambda\cdot
        \alpha^2}\cdot \pp{\alpha^{m-1} - \lambda}^{-1}$ belongs to the
        localised ideal $\DD{Ill}{\DIl_{\Dmm}}$. Consequently $\lambda\alpha -
        m!\cdot\partial$ belongs to this
        ideal and $\DIll = \pp{\pp{\lambda\alpha - m!\cdot\partial} +
        J}_{\Dmm}$. This proves
        that
        \[
        \pp{\frac{\DTpoly}{\DIl}}_{\Dmm}  \simeq
        \pp{\frac{\DTpoly}{\pp{\lambda\alpha - m!\cdot\partial} + J}}_{\Dmm}  \simeq
        \frac{\DTring}{\annn{\dual{S}}{f}}.
        \]
        Next we look near $\Dmm = (0,\dots, \omega)$. Here $\alpha^2$ is invertible,
        thus $\alpha^{m-1} - \lambda$ belongs to $\DIll$, then
        $\DIll = \left( \alpha^{m-1} - \lambda, \annn{\dual{S}}{\partial\hook
        f} \right)$ and the localised algebra is isomorphic to the apolar
        algebra of $\partial\hook f$.
    \end{proof}

    \subsection{Smoothability of algebras with Hilbert function $(1, 5, 4, 1)$.}

    For simplicity and brevity of presentation we will use the fact that algebras of length at most
    $10$ are smoothable, see \cite{CN10}. In fact it is known that algebras of
    length $11$ are smoothable, so the sceptic reader should replace $(1, 5,
    4, 1)$ with~e.g.~$(1, 7, 4, 1)$. The crucial part of the proof is
    Proposition~\ref{ref:squares:prop}, allowing us to compute the apolar
    ideal and use Corollary~\ref{ref:bjoerktactcor:cor}.

    The proof is straightforward once we use the following inductive lemma:
    \begin{lem}
        Let $A  \simeq \DTring/I$ be a Gorenstein algebra of socle
        degree $j\geq 3$ with Hilbert function
        decomposition containing the term $\Dhdvect{A,j-2} = (0, e, 0)$, where
        $e > 0$.
        
        The scheme
        $\Spec A$ is a flat degeneration of schemes isomorphic to $\DD{dp}{\Spec
        k} \sqcup \Spec
        B$, where $B$ is a local Gorenstein $k$-algebra of socle degree
        $j$,
        which has the same
        Hilbert function decomposition as $A$ except for the term $\Dhdvect{B, j-2} = (0, e-1,
        0)$.
    \end{lem}

    \begin{proof}
        \def\Dgp{f}
        Let $g\in k[x_1,\dots,x_n]$ be the  dual generator from
        Proposition~\ref{ref:squares:prop}, then we can write $g = \Dgp +
        x_1^2$, where
        $\Dgp\in k[x_2,\dots,x_n]$. Clearly there exists $\partial\in \DTring$ such
        that $\partial\hook f = 1$, then $\partial^2\hook \Dgp = 0$ and we
        may apply Corollary~\ref{ref:bjoerktactcor:cor} to $g = \Dgp +
        x_1^2\cdot \partial\hook f$.
        It follows that the
        apolar algebra of $g$ is a degeneration of union of the apolar algebra of
        $\partial\hook \Dgp = 1$ -- which is isomorphic to $\Ddp$ -- and the apolar algebra $B$ of
        $\Dgp$.
        The assumptions on $B$ except for the Hilbert function decomposition
        are straightforward. For the Hilbert function decomposition, see
        Proposition~\ref{ref:hilbertfunctionfrompolynomial:prop}.
    \end{proof}

    \subsection{Smoothability of algebras with Hilbert function $(1, 4, 4, 3,
    1)$.}

    In this subsection we have two main points: proving irreducibility, which is
    easy once we apply secant varieties, and proving the existence of a
    smooth, smoothable point of the Hilbert scheme.

    \begin{prop}\label{ref:usingsecants:prop}
        The projective set of quartics in three variables, whose
        apolar algebras have the
        Hilbert function $(1, n, m, n, 1)$ with $m\leq 3$, is
        the third secant variety to the
        fourth Veronese embedding of $\mathbb{P}^2$. In particular it is
        irreducible.
    \end{prop}
    \begin{proof}
        See~e.g.~\cite{LO}, which refers to other sources.
    \end{proof}

    \begin{prop}\label{ref:h14431irred:prop}
        The Gorenstein algebras with Hilbert function $(1, 4, 4, 3, 1)$ are
        parametrised, up to isomorphism, by an irreducible
        Zariski-constructible subset $V
        \subseteq \DSS^{\leq 3} = \DSring^{\leq 3}$, where we view $\DD{Affthree}{\DSS^{\leq 3}}$ as an affine
        space.
    \end{prop}

    \begin{proof}
        The only possible decomposition of the Hilbert function is
        $\Dhdvect{0} = (1, 3, 3,
        3, 1); \Dhdvect{1} = (0, 1, 1, 0)$. Let us fix a dual socle generator in the
        standard form $f = f_4 + f_3 + f_2$. Corollary~\ref{ref:topdegreefirstrow:cor}
        implies that the apolar algebra of $f_4$ has Hilbert function
        $\Dhdvect{0}$, thus $\Dhdvect{0}$ depends only on $f_4$. By Proposition~\ref{ref:usingsecants:prop}
        forms $f_4$ whose apolar algebras have Hilbert function $(1, 3, 3, 3,
        1)$ constitute a Zariski locally closed irreducible subset $Z \subseteq\DSS^4$.

        Once we know that $h(1) = 4$ and that the first row of the
        decomposition is $(1, 3, 3, 3, 1)$, then $h = (1, 4, 4, 3, 1)$ is
        maximal among possible Hilbert functions. Proposition
        \ref{ref:maximalopen:prop} applied to $Z \times \DSS^{\leq 3}$ proves
        that the set of dual socle generators of algebras of the form $(1, 4,
        4, 3, 1)$ is open in $Z \times \DSS^{\leq 3}$, thus irreducible.
    \end{proof}

    \begin{prop}\label{ref:h14431smoothel:prop}
        Let $n = 4$.
        The subset $V \subseteq \DAffthree$ defined in
        Proposition~\ref{ref:h14431irred:prop} induces a morphism
        $V\to \DHilbGPn$, whose image is irreducible.
        The image of the point $p$ corresponding to the apolar algebra of
        \[\DD{tmppoly}{x_1^4 + x_2^4 + x_3^4 + x_4^2\cdot (x_1 + x_2)}\]
        is a smooth point of $\DHilbGPn$ lying
        in $\DHilbGPnzero$.
    \end{prop}

    \begin{proof}
        The first claim follows from Proposition
        \ref{ref:familiesofpolystohilb:prop} and the fact that the image of an
        irreducible set is irreducible.
        The Hilbert function computation is most conveniently conducted using
        Proposition~\ref{ref:hilbertfunctionfrompolynomial:prop}, we leave it
        to the reader.

        Next, one should check that $\Dtmppoly$ satisfies assumptions of
        Corollary~\ref{ref:bjoerktactcor:cor} with respect to the variable
        $\Dx:=x_4$. The resulting flat family presents
        the apolar algebra of $\Dtmppoly$ as a flat limit of schemes
        isomorphic to $W$, where $W$ is a disjoint union of a double point and
        an apolar algebra of $x_1^4 + x_2^4 + x_3^4$. Then one may continue by
        applying Corollary~\ref{ref:bjoerktactcor:cor} trice or using the fact
        that every finite Gorenstein algebra $A$ with $h_A(1) \leq 3$ is smoothable,
        see \cite{CN10} and references therein.
        To calculate that $p\in \DHilbGPnzero$ is smooth it sufficies to prove
        that the tangent space at $p$ has $k$-rank $4\cdot 13 = 52$. The most
        straightforward way to do this is to calculate the apolar ideal of
        $\Dtmppoly$ together with its square and use Proposition
        \ref{ref:Gorensteintangentspace:prop}. Since this is pure computation
        (but see Remark \ref{ref:tangentbjorkoj:remark} below), we leave it to the reader.
    \end{proof}

    \begin{remark}\label{ref:tangentbjorkoj:remark}
        If $k[t] \to \DTT[t]/I_t$ is the flat family from Theorem
        \ref{ref:bjorktactical:thm}, then one may analyze flatness of
        $k[t] \to \DTT[t]/I_t^2$ using the same theorem, obtaining a criterion
        for flatness in terms of $I_t$. Unfortunately in general its check seems as difficult as making a
        direct computation.
    \end{remark}

    \subsection{Smoothability of graded algebras with Hilbert function $(1, 3,
    3, 3, 3, 1)$}

    As mentioned in the proof of Proposition~\ref{ref:h14431smoothel:prop},
    every finite Gorenstein algebra $A$ satisfying $h_A(1) \leq 3$ is
    smoothable; the proof depends on a structure theorem on the resolutions of
    such $A$. As an example of a more general method, we present a different approach to this smoothability,
    which shows a tight connection between finding equations of secant varieties and
    proving smoothability of certain Gorenstein algebras.

    Since we are analyzing graded algebras, we may take homogeneous forms of
    degree five as the dual socle generators. The following proposition shows that the set of
    such forms is irreducible, in fact describing a general element.

    \begin{prop}\label{ref:usingsecantstwo:prop}
        The projective set of quartics in three variables, whose apolar
        algebras have the
        Hilbert function $(1, n, m, m, n, 1)$ with $m\leq 3$, is
        the third secant variety to the
        fifth Veronese embedding of $\mathbb{P}^2$. In particular it is
        irreducible.
    \end{prop}
    \begin{proof}
        See \cite[Thm 3.2.1]{LO}.
    \end{proof}

    A general form of this set is equal, up to a linear change of
    coordinates, to the form $x_1^5 + x_2^5 + x_3^5\in k[x_1, x_2, x_3]$. The
    apolar algebra of such form is smoothable by Corollary \ref{ref:bjoerktactcor:cor}.

    \bibliographystyle{alpha}

    \section*{Acknowledgements}

    The author is very grateful to Jaros\l{}aw Buczy\'nski for introducing him
    to the
    research-level algebraic geometry, many hours spent on teaching and
    explaining, his throughout assistance and sense of humour. Thanks to you
    the master degree was a happy research time!

    The topic of this work was primarily inspired by the work of
    Gianfranco Casnati and Roberto Notari, who also introduced the author to many
    folklore facts during the scientific visits in Torino. Thank you for all!

%    \bibliography{refs}

    \end{document}